\documentclass[10pt]{article}  
\usepackage[all]{xy}
\usepackage[mathscr]{euscript}
\usepackage{amsmath, amsfonts, amssymb, mathrsfs}
\usepackage{appendix}
\usepackage{amsthm}
\usepackage{comment}
\usepackage{textcomp}
\usepackage{setspace}
\usepackage{fix-cm}
\usepackage{mathabx}
\usepackage{enumerate}

 \usepackage{mathtools}
 \usepackage{nccmath}

\usepackage{xcolor}

\usepackage[alphabetic,lite]{amsrefs}

\newcommand{\bC}{\mathbb C}
\newcommand{\bF}{\mathbb F}

\newcommand{\bZ}{\mathbb Z}

\newcommand{\bQ}{\mathbb Q}

\newcommand{\frakM}{\mathfrak M}

\newcommand{\frakg}{\mathfrak g}

\newcommand{\phibar}{\overline{\varphi}}
\newcommand{\rhobar}{\bar{\rho}}

\newcommand{\al}{\alpha}

\newcommand{\lam}{\lambda}
\newcommand{\varep}{\varepsilon}

\DeclareMathOperator{\Hom}{Hom}

\DeclareMathOperator{\GL}{GL}
\DeclareMathOperator{\PGL}{PGL}
\DeclareMathOperator{\SL}{SL}

\DeclareMathOperator{\SU}{SU}
\DeclareMathOperator{\GU}{GU}
\DeclareMathOperator{\U}{U}

\DeclareMathOperator{\id}{id}

\DeclareMathOperator{\Gal}{Gal}

\DeclareMathOperator{\Ind}{Ind}

\theoremstyle{plain}
\input xy
\xyoption{all}
\newtheorem{Thm}{Theorem}[section]
\newtheorem{Lem}[Thm]{Lemma}
\newtheorem{Prop}[Thm]{Proposition}
\newtheorem{Cor}[Thm]{Corollary}
\newtheorem{Def}[Thm]{Definition}
\theoremstyle{Rem}
\newtheorem{Rem}[Thm]{Remark}

\numberwithin{equation}{section}
\numberwithin{paragraph}{section}

\newtheorem{what}{Question}
\newenvironment{rem}{\begin{Rem} \rm}{\end{Rem}}
\newenvironment{Deff}{\begin{Def} \rm}{\end{Def}}
\newenvironment{What}{\begin{what} \rm}{\end{what}}

\def\1{\mathbf{1}}
\def\Hom{\mathrm{Hom}}
\def\Gal{\mathrm{Gal}}

\def\CC{\mathbb{C}}

\def\FF{\mathbb{F}}
\def\FFF{\overline{\mathbb{F}}}
\def\GG{\mathbb{G}}

\def\MM{\mathfrak{M}}

\def\OO{\mathcal{O}}

\def\QQ{\mathbb{Q}}
\def\QQQ{\overline{\mathbb{Q}}}

\def\SSSS{\mathfrak{S}}

\def\ZZ{\mathbb{Z}}

\def\Q{{\mathbb Q}}
\def\F{{\mathbb F}}

\def\Qp{{{\mathbb Q}_p}}

\def\Z{{\mathbb Z}}

\def\C{{\mathbb C}}

\DeclareMathOperator{\BC}{BC}
\DeclareMathOperator{\WD}{WD}

\DeclareMathOperator{\St}{St} %Steinberg
\DeclareMathOperator{\Art}{Art} 
\DeclareMathOperator{\Deform}{Def} 
\DeclareMathOperator{\Frob}{\varphi} 

\newcommand{\der}{G^{\text{der}}} %derived subgroup of G
\newcommand{\conj}{c} %nontrivial element of Gal(Qp^2/Qp)
\newcommand{\Oe}{\mathcal{O}}  %ring of integers of E
\newcommand{\rese}{k} %res field of E
\newcommand{\Profin}{\Gamma}
\newcommand{\Inertia}{\mathcal{I}}
\newcommand{\AbGal} {\Gamma}%absolute Galois group
\newcommand{\dualC}{^C\hat{G}}
\newcommand{\Lpacket}{\Pi} % name for L packet in Langlands correspondence
\newcommand{\Functor}{T_{\mathrm{dd}}^{\ast}}
\newcommand{\weyl}{s_0} %nontrivial element of the Weyl group
\usepackage{color}
\usepackage{marginnote}

\begin{document}  

\title{From $p$-modular to $p$-adic Langlands correspondences for $\U(1,1)({\QQ_{{p}^{2}}}/\QQ_{p})$: deformations in the non-supercuspidal case}
\date{}
\author{R. Abdellatif, A. David, B. Romano, H. Wiersema}
\maketitle
\thispagestyle{empty}
\begin{abstract}
This paper surveys what is known about (conjectural) $p$-adic and $p$-modular semisimple Langlands correspondences in the non-supercuspidal setting for the unramified quasi-split unitary group $\U(1,1)(\QQ_{p^{2}}/\QQ_{p})$. It focuses in particular on the potential of deformation theory to relate these correspondences.
\end{abstract}

\tableofcontents

\section{Introduction}\label{section-intro}
Langlands correspondences are non-abelian, and {mostly still} conjectural, generalisations of class field theory. Recall that, given a non-Archimedean local field $F$ with positive residual characteristic $p$ and separable closure $\overline F$, the local version of class field theory provides a natural identification of continuous $\C$-valued characters of $\Gamma_{F} := \Gal(\overline{F}/F)$ with smooth complex characters (i.e. irreducible smooth representations) of $F^{\times} = \GL_{1}(F)$. In 1967, Langlands conjectured that {{some}} analogous correspondence should exist for higher-dimensional representations (characters being one-dimensional representations), and that $n$-dimensional continuous representations of $\Gamma_{F}$ should naturally correspond to some admissible smooth representations of $\GL_{n}(F)$. He even went further, as he conjectured later that such a statement should hold for reductive groups $\GG$ other than $\GL_{n}$ if one prescribes certain conditions on the image of the Galois representations involved in such correspondences, and if one allows correspondences with finite fibres that are not necessarily one-to-one. More specifically, we should be able to gather the relevant representations of $G = \GG(F)$ into disjoint sets called packets, each of which should correspond to a single Galois representation. For more details and explanations of this so-called classical setting, in which the representations of both $G$ and $\AbGal_F$ are defined over $\bC$, the reader should for instance refer to \cite{Borel76}. 

In the last decades, congruences between modular forms as well as deformations of $p$-modular Galois representations have played important roles in the proofs of some major arithmetical results, such as Wiles' proof of Shimura--Taniyama--Weil conjecture \cite{Wiles} and Kisin's work on the Fontaine--Mazur conjectures \cite{KisinFMC}. These advances motivate in turn the search for analogues of Langlands correspondences that classify representations with coefficients in rings other than $\C$. The latest examples are given by $p$-modular Langlands correspondences, which take coefficients in an algebraically closed field of characteristic $p$, and $p$-adic Langlands correspondences, which take coefficients in a (large enough) finite extension of $\bQ_p$. In both of these settings, while we can often understand the Galois side of the correspondence, the so-called automorphic counterpart is still very mysterious \cite{AHHV}, even when studying representations of $\GL_{2}(F)$ \cite{BP12, Le, Wu}.

A natural question in this setting is how compatible the $p$-adic and $p$-modular statements are. More precisely, if $E$ is a finite extension of $\bQ_p$ with ring of integers $\Oe$, {maximal ideal $\mathfrak{p}$}, and residue field $k = \Oe/\mathfrak{p}$, then any representation $\pi$ {defined over $\Oe$} naturally gives {(via reduction modulo $\mathfrak{p}$)} a representation $\overline\pi$ over $k$. On the other hand, given a representation over $k$, deformation theory allows us to study representations over $\Oe$ whose reduction modulo $\mathfrak{p}$ is isomorphic to $\overline\pi$. 
The ability to move between characteristic $p$ and characteristic zero naturally leads to the following kind of functoriality problem: let $\overline{k}$  be an algebraically closed field of characteristic $p$ and let $\overline{\pi}$ be an irreducible smooth $\overline{k}$-representation of a $p$-adic group $G$ that corresponds (under an appropriate $p$-modular Langlands correspondence) to a continuous representation $\overline\sigma$ of $\Gamma_{F}$. Is it possible to relate deformation theories for $\overline\sigma$ and $\overline\pi$ under an appropriate $p$-adic Langlands correspondence? Equivalently, how do reduction modulo $p$ and deformation theory help to connect the aforementioned $p$-modular and $p$-adic Langlands correspondences? Note that the ability to answer these questions through deformation theory is a keystone in Colmez's proof of $p$-adic Langlands correspondences for $\GL_{2}(\QQ_{p})$ for $p \geq 5$ \cite{Col10, Kis10}, so it is natural to consider them when interested in $p$-adic Langlands correspondences for other $p$-adic groups.

To our knowledge, this has not been studied much besides the $\GL_{n}(F)$ case, even in the $\ell$-adic case (i.e. when $\overline{k}$ is an algebraically closed field of positive characteristic $\ell \not= p$, see \cite{VignerasLivre}), though the relevant deformations in this setting are quite well understood on the Galois side \cite{BellovinGee, BoPat20}. In a current work in progress, the two first authors study the case of special linear group $\SL_{2}(\Q_{p})$, which is the first group for which a semisimple $p$-modular Langlands correspondence involving actual packets has been proved \cite{Abdellatif12}. The present paper focuses on the case of the quasi-split unramified unitary group $G = \U(1,1)(\QQ_{p^{2}}/\QQ_{p})$, which is the first non-split group for which a semisimple $p$-modular Langlands correspondence has been settled \cite{Koziol}. Note that this correspondence also involves actual packets: it is not a one-to-one, but a finite-to-one, correspondence. Our question is the following: how does this semisimple correpondence behave under deformations, i.e. when the objects it involves are lifted to $p$-adic representations/parameters?

This paper explains what is known so far in this direction for non-supercuspidal objects. As above, let $G = \U(1,1)(\QQ_{p^{2}}/\QQ_{p})$. First, in Section \ref{section-repsofG}, we review basic definitions about representations of $p$-adic groups in the context of the group $G$, including the definition of a non-supercuspidal representation of $G$.
We finish the section with the classification of irreducible smooth non-supercuspidal representations of $G$ over $\FFF_{p}$, where $\FFF_{p}$ denotes an algebraic closure of the residue field of $\Q_{p}$. In Section \ref{section-langlands} we introduce Langlands parameters and describe where non-supercuspidal representations fit into the semisimple Langlands correspondence Kozio\l \ attached to representations of $G$ \cite{Koziol}. In Section \ref{section-nonsc-deformations} we begin to explore how the semisimple Langlands correspondence behaves when lifted to characteristic $0$. To do so, we describe recent results of Hauseux--Sorensen--Schmidt \cite{HSS1, HSS2} about the behaviour of parabolic induction under deformation, and we explicitly determine which representations of $G$ their results apply to on the automorphic side. 

Finally, in Section \ref{SectionDeformationParameters} we use recent work of Kozio\l--Morra to study how the Galois side of the correspondence behaves under deformation \cite{KM}. Our method here is to use the theory of Kisin modules to better understand the deformations of interest. 
The distinction between supercuspidal and non-supercuspidal representations has a counterpart on the Galois side of the correspondence, but we do not need to make this distinction for the main results of Section~\ref{SectionDeformationParameters}.
We can afford to be more general in this section because the results on the Galois side are uniform. (In contrast, the construction and deformations of supercuspidal representations on the automorphic side involves different techniques than for their non-supercuspidal counterparts.)
We point out open problems related to deformations on both sides of the Langlands correspondence that are the subject of work in progress.

\subsection*{General notation}
\label{section-notation}
We fix a prime integer $p$. We let $\QQ_{p}$ denote the field of $p$-adic numbers and we fix a separable closure $\overline\bQ_p$ of $\bQ_p$, as well as an algebraic closure $\FFF_{p}$ of the residue field $\FF_{p}$ of $\QQ_{p}$. {We let $\overline{\ZZ}_{p}$ denote the ring of integers of $\overline{\QQ}_{p}$.} Given any positive integer $n$, we let $\bQ_{p^n}$ be the unique unramified extension of degree $n$ of $\bQ_p$ in $\overline\bQ_p$ and we denote by $\Z_{p^{n}}$ its ring of integers. We fix an isomorphism $\ZZ_{p^{n}}/p\ZZ_{p^{n}} \to \bF_{p^n} \subset \FFF_{p}$ {identifying the residue field of $\bQ_{p^n}$ with $\bF_{p^n}$,} and we write $\overline x \in \FF_{p^{n}}$ for the image of $x \in \ZZ_{p^{n}}$ under the composite map $\ZZ_{p^{n}} \twoheadrightarrow \ZZ_{p^{n}}/p\ZZ_{p^{n}} \to \bF_{p^{n}}$. {These maps extend to a} reduction map $\overline{\ZZ}_{p} \to \FFF_{p}$, which we also denote $x \mapsto \overline x$, {and which allows us to identify the residue field of $\overline{\bQ}_p$ with $\overline{\bF}_p$.}

Let $\conj$ denote the nontrivial element of $\Gal(\bQ_{p^2}/\bQ_p)$ and, for $g = \begin{psmallmatrix} x & y\\z & w \end{psmallmatrix} \in \GL_2(\bQ_{p^2})$, write $\conj(g)$ for $\begin{psmallmatrix} \conj(x) & \conj(y)\\  \conj(z) &  \conj(w) \end{psmallmatrix}$. We write $\GG$ for the unramified quasi-split unitary group in two variables defined over $\Q_{p}$ and set $G = \GG(\bQ_p)$. As usual, and more concretely, we identify $G = \U(1,1)(\Q_{p^{2}}/\Q_{p})$ with the following subgroup of $\GL_2(\bQ_{p^2})$, where $g^* = \conj(g)^t$ denotes the conjugate transpose of $g \in \GL_{2}(\QQ_{p^{2}})$:
\begin{equation*}
G = \left\{ g \in \GL_2(\bQ_{p^2}) \mid g^*\begin{pmatrix}0 & 1\\1 & 0\end{pmatrix}g = \begin{pmatrix}0 & 1\\1 & 0\end{pmatrix}  \right\} \ .
\end{equation*} 
We write $I_{2} = \begin{psmallmatrix} 1 & 0 \\ 0 & 1 \end{psmallmatrix}$ for the $2 \times 2$ identity matrix.
We let $B$ be the subgroup of upper-triangular matrices in $G$. It is a Borel subgroup of $G$ with Levi decomposition $B = TU$, where $U$ denotes the unipotent radical of $B$ and $T$ the maximal torus of $G$ made of all diagonal matrices in $G$:
\begin{equation*}
T = \left\{ t(x) := \begin{pmatrix} x & 0\\0 & \conj(x)^{-1} \end{pmatrix}, \  x \in \bQ_{p^2}^\times \right\}.
\end{equation*}
Note that $U$ is an abelian group, isomorphic to the additive group $(\QQ_{p}, +)$. We let $B^{-}$ be the opposite Borel to $B$ with respect to $T$.
Note that $B^{-}$ is nothing but the subgroup of lower-triangular matrices in $G$.

We let $W$ be the Weyl group of $G$. Recall that we have $W = N_{G}(T)/T$, where we write $N_{G}(T)$ for the normaliser of $T$ in $G$. The group $W$ has two elements, and we let ${\weyl}$ denote the non-trivial one. A representative of ${\weyl}$ in $N_{G}(T)$ is given by the matrix $s = \begin{psmallmatrix} 0 & 1 \\ 1 & 0 \end{psmallmatrix}$. 

We let $\der = \SU(1, 1)(\bQ_{p^2}/\bQ_p)$ denote the derived group of $G$: it consists of all matrices in $G$ with determinant $1$. Recall that, once we fix an element $\varepsilon \in \Q_{p^{2}}$ such that $\Q_{p^{2}} = \Q_{p}(\varepsilon)$ and $c(\varepsilon) = -\varepsilon$, conjugation by the element $\begin{psmallmatrix} \varepsilon & 0\\0 & 1 \end{psmallmatrix}$ defines a isomorphism from $\der$ to $\SL_{2}(\Q_{p})$. 

Given any ring $R$, any character $\chi: H \to R^{\times}$ of a group $H$, and any positive integer $n$, we write $\chi^{n}: H \to R^{\times}$ for the $R$-valued character of $H$ given by {{$\chi^{n}(h) := \chi(h)^{n}$}}. For any group $H$, we write $\1_H$ for the trivial representation of $H$ over $\overline\bF_p$.

Given any finite extension $F/\bQ_p$, we write $\AbGal_F := \Gal(\overline\bQ_p/F)$ for its absolute Galois group and $\Inertia_F \subset \AbGal_F$ for the inertia subgroup of $\AbGal_F$. We fix an element $\Frob \in \Gamma_{\bQ_p}$ such that the image of $\Frob$ in the abelianisation $\AbGal_{\bQ_p}^{\text{ab}}$ is equal to the image of $p^{-1}$ under the reciprocity map $\bQ_p^\times \to \AbGal_{\bQ_p}^{\text{ab}}$ of local class field theory. 
It is useful to recall that we have the following short exact sequence, where $W_{\bQ_p} \subset \Gamma_{\bQ_p}$ denotes the Weil group of $\bQ_p$ and $\langle \Frob \rangle \simeq \Z$ denotes the subgroup of $\Gamma_{\bQ_p}$ generated by $\Frob$:
\begin{equation*}
%\label{SesWeilGroup}
1 \longrightarrow {\Inertia}_{\bQ_p} \longrightarrow W_{\bQ_p} \longrightarrow \langle \Frob \rangle \longrightarrow 1 \ .
\end{equation*}

\section{Non-supercuspidal representations of $G$ over $\overline{\bF}_p$}
\label{section-repsofG}
In this section, we recall the classification of irreducible smooth non-supercuspidal representations of $G = \U(1, 1)(\bQ_{p^2}/\bQ_p)$ over $\FFF_{p}$, as first established in \cite[Chapitre 5]{Abdellatif}. Given a commutative ring $A$, recall that a representation $(\pi, V)$ of $G$ on an $A$-module $V$ is called {\it smooth} if every element of $V$ has open stabiliser in $G$. Equivalently, this means that $V$ can be written as $\displaystyle \bigcup_{K} V^{K}$, where the union is taken over the open subgroups $K$ of $G$, and where $V^{K} := \{ v \in V \mid \pi(k)v = v \text{ for all } k \in K\}$ denotes the subspace of $K$-fixed vectors in $V$. We write $\mathrm{Mod}_{G}^{\infty}(A)$ for the category of smooth representations of $G$ over $A$, and define similarly $\mathrm{Mod}_{\Gamma}^{\infty}(A)$ for any topological group $\Gamma$.

Given any closed subgroup $H$ of $G$ and any smooth $A[H]$-module $V$, we set 
\begin{eqnarray*}
\nonumber \Ind_{H}^{G}(V) &=& \{f : G \to V \mid \text{ there exists a compact open subgroup } K_{f} \subset G \text{ such that }\\ \nonumber & & f(hgk) = h \cdot f(g) \text{ for all } h \in H, g \in G, k \in K_{f}\}.
\end{eqnarray*}
The $A$-module $\Ind_{H}^{G}(V)$ is naturally endowed with a smooth action of $G$ by right translations, and the corresponding functor $\Ind_{H}^{G} : \mathrm{Mod}_{H}^{\infty} \to \mathrm{Mod}_{G}^{\infty}$ is called {\it smooth induction from $H$ to $G$}. A particular case of smooth induction is given by {\it parabolic induction}, when $H$ is a (proper) parabolic subgroup of $G$. In our setting, $G$ is of relative rank $1$ over $F$, hence if $H$ is a proper parabolic subgroup, we can assume without loss of generality that $H = B$ is the Borel subgroup defined above. Parabolic induction is then defined as follows: given a smooth character $\chi : T \to A^{\times}$, we can inflate $\chi$ to a character of $B$ trivial on $U$ (still denoted by $\chi$) and consider $\Ind_{B}^{G}(\chi)$.
\begin{rem} 
Note that $\Ind_{B}^{G}$ actually defines a functor $\mathrm{Mod}_{T}^{\infty, adm}(A) \to \mathrm{Mod}_{G}^{\infty, adm}(A)$, where we write $\mathrm{Mod}^{\infty, adm}_{X}(A)$ for the full subcategory of $\mathrm{Mod}^{\infty}_{X}(A)$ whose objects are {\it admissible} representations, i.e. such that $V^{K}$ is an $A$-module of finite type for any compact open subgroup $K$ of $G$.
\end{rem}

When classifying irreducible smooth representations of $G$ over $A = \FFF_{p}$, a natural first step is to study parabolically induced representations, not only because of their uniform and easy construction, but also because they reflect representations of $G$ that come from (proper) reductive subgroups of $\GG$. This leads to the following definition.  
\begin{Deff}
\label{DefnonspcuspG}
Let $\pi$ be an irreducible smooth representation of $G$ over $A$.
\begin{itemize}
\item We say that $\pi$ is {\it non-supercuspidal} if $\pi$ is isomorphic to a subquotient of $\mathrm{Ind}_{B}^{G}(\chi)$ for some smooth character $\chi : T \to A^{\times}$.
\item We say that $\pi$ {\it a principal series representation} if $\pi$ is isomorphic to a parabolically induced representation $\mathrm{Ind}_{B}^{G}(\chi)$ for some smooth character $\chi : T \to A^{\times}$. 
\end{itemize}
\end{Deff}
Following \cite[Th\'eor\`eme 5.1.2]{Abdellatif}, we know that isomorphism classes of irreducible smooth non-supercuspidal representations of $G$ over $\FFF_{p}$ split into three disjoint families: characters, principal series representations, and special series representations. We now recall the explicit description of these objects, as stated in \cite[Theorem 4.3]{Koziol}. This requires us to introduce the following notation.
For any $\lambda \in \FFF_{p}^{\times}$, we denote by $\mu_{\lambda} : \Q_{p^{2}} \to \FFF_{p}^{\times}$ the smooth character trivial on $\ZZ_{p^{2}}^{\times}$ that maps $p$ to $\lambda$. Also, we let $\omega : \Q^{\times}_{p^{2}} \to \FFF_{p}^{\times}$ be the smooth character such that $\omega(p) = 1$ and $\omega(u)=\overline{u} \in \FF_{p^{2}}^{\times}$  is the reduction
modulo $p$ of $u \in \Z_{p^{2}}^{\times}$. 

\subsection{Principal series representations and characters of $G$}
As $T$ is naturally isomorphic to $\Q_{p^{2}}^{\times}$ via the group homomorphism $\Q_{p^{2}}^{\times} \to T$ that maps $x$ to $t(x)$, all characters $\mu_{\lambda}$, for $\lambda \in \FFF_{p}^{\times}$, and $\omega$ can (and will) also be seen as smooth characters of $T$. It is then straightforward to check that any smooth character of $T$ is of the form $\mu_{\lambda}\omega^{r}$ for a unique pair $(r, \lambda) \in \bZ \times \FFF_{p}^{\times}$ with $0 \leq r \leq p^{2} - 2$. By \cite[Th\'eor\`eme 5.3.1, (2)]{Abdellatif}, a parabolically induced representation $\Ind_{B}^{G}(\chi)$ is irreducible if and only if $\chi$ does not extend to a smooth character of $G$, so to describe the principal series representations of $G$ we must identify which of these characters can be extended to smooth characters of $G$.
\begin{Lem}
\label{lemCharG}
Let $(r, \lambda) \in \bZ \times \FFF_{p}^{\times}$ with $0 \leq r \leq p^{2} - 2$. The smooth character $\mu_{\lambda}\omega^{r} : T \to \FFF_{p}^{\times}$ extends to a smooth character $\chi : G \to \FFF_{p}^{\times}$ if, and only if, $\lambda = 1$ and $r = m(p-1)$ with $0 \leq m \leq p$. In this case, we have $\chi = \omega^{- m} \circ \det$. 
\end{Lem}
\begin{proof}
Let $\mu_{\lambda}\omega^{r}$ be a smooth character of $T$ that extends to a smooth character $\chi$ of $G$. Then $\chi$ must be trivial on the derived group $\der$ of $G$, hence it is trivial on the diagonal torus of $\der$, which means that we must have $\chi(x) = 1$ whenever $\conj(x) = x$, i.e. whenever $x$ belongs to $\Q_{p}^{\times}$. As $p$ belongs to $\Q_{p}^{\times}$, we must have $\chi(t(p)) = (\mu_{\lambda}\omega^{r})(p) = 1$, i.e. $\lambda = 1$.

Now let $\zeta$ be a root of unity of order $p^{2} - 1$ in $\Q_{p^{2}}^{\times}$: then $\zeta$ lies in $\Z_{p^{2}}^{\times}$ and $\omega(\zeta)$ generates the multiplicative group $\FF_{p^{2}}^{\times}$. Since $\zeta^{p+1}$ belongs to $\Q_{p}^{\times}$, we must have $\chi(t(\zeta^{p+1})) = 1$, i.e. $\zeta^{r(p+1)} = 1$. This implies, as $\zeta$ is of order $p^{2} -1$, that $p-1$ divides $r$, hence $r = m(p-1)$ with $0 \leq m \leq p$. Conversely, if $r$ satisfies such a condition, then $\omega^{r} = (\omega^{p-1})^{m}$ is the restriction of the smooth character $\omega^{-m} \circ \det : G \to \FFF_{p}^{\times}$, since for any $k$ we have $(\omega^{k} \circ \det)(t(\zeta)) = \omega^{k}(\zeta\conj(\zeta)^{-1}) = \omega^{k}(\zeta\zeta^{-p}) = \omega^{k}(\zeta^{1-p}) = \zeta^{(1-p)k}$.
\end{proof}

\begin{Cor}
\label{CorPrincipalSeries}
Let $(r, \lambda) \in \bZ \times \FFF_{p}^{\times}$ with $0 \leq r \leq p^{2} -2$. If $(r, \lam) \neq ( (p - 1)m, 1)$ for all $0 \leq m \leq p$, 
then $\Ind_{B}^{G}(\mu_{\lambda}\omega^{r})$ is an irreducible smooth representation of $G$. 
\end{Cor}
\begin{proof}
This is a direct application of \cite[Th\'eor\`eme 5.3.1, (2)]{Abdellatif} using Lemma \ref{lemCharG}.
\end{proof}

\subsection{Special series representations} \label{section-steinberg}
Now assume that $\chi$ is a smooth character of $T$ that extends to a smooth character $\omega^{r} \circ \det$ of $G$. Up to twisting by $\omega^{-r} \circ \det$, we can assume that $\chi$ is trivial and focus on the parabolically induced representation $\Ind_{B}^{G}(\1_B)$. We know from \cite[Th\'eor\`eme 5.3.10, (2)]{Abdellatif} that $\Ind_{B}^{G}(\1_B)$ is a reducible representation of $G$ of length $2$, with the trivial character $\1_G$ as subrepresentation. The irreducible quotient representation
\begin{equation*}
\St_{G} := \Ind_{B}^{G}(\1_{B})/\1_{G}
\end{equation*}
is called the {\it Steinberg representation}, and twists of Steinberg representations by smooth characters of $G$ are called {\it special series representations}. Such representations are exactly those that appear as irreducible quotients of $\Ind_{B}^{G}(\omega^{k} \circ \det)$ for any integer $0 \leq k \leq p$.

\subsection{Classification of non-supercuspidal representations of $G$}
If we combine the results of the previous subsections, we recover the following classification of non-supercuspidal representations of $G$, as given in \cite[Theorem 4.3]{Koziol}. 
\begin{Thm}
\label{ClassificationNonSupercuspidal}
Let $V$ be an irreducible smooth, non-supercuspidal representation of $G$. Then $V$ is isomorphic to one and only one of the following representations:
\begin{itemize}
\item $\omega^{k} \circ \det$, with $0 \leq k \leq p$; 
\item $(\omega^{k} \circ \det) \otimes \St_{G}$, with $0 \leq k \leq p$;
\item $\Ind_{B}^{G}(\mu_\lam\omega^r)$ with $(r, \lam) \in \bZ \times \overline{\bF}_p^\times$ such that $0 \leq r \leq p^{2} - 2$ and $(r, \lam) \neq ((p - 1)m, 1)$ for all $0 \leq m \leq p$.
\end{itemize}
\end{Thm}
\begin{proof}
We saw above that all these representations are irreducible smooth non-supercuspidal representations of $G$ over $\FFF_{p}$. We know from \cite[Lemme 5.4.1]{Abdellatif} (resp. \cite[Lemme 5.4.3]{Abdellatif}) that principal series (resp. special series) representations coming from distinct characters are non-isomorphic. As principal series representations and special series representations are infinite dimensional, they cannot be isomorphic to smooth characters of $G$. Finally, \cite[Lemme 5.5.1, Corollaire 5.5.4]{Abdellatif} ensures that a principal series representation cannot be isomorphic to a special series representations as their spaces of $I(1)$-invariant vectors have different dimensions (as vectors spaces over $\FFF_{p}$), where $I(1)$ denotes the standard pro-$p$-Iwahori subgroup of $G$.
\end{proof}

\begin{rem}
The classification of supercuspidal representations of $G$ is also known, as it was proven by Kozio\l \ \cite[Theorem 5.7]{Koziol}. Note that it is one of the only three cases where such a classification is explicit (see \cite{Breuil01} and \cite{Abdellatif12} for the two other cases). We choose to not recall this part of the classification in this paper as it would require a lot of extra material that we do not use at all in the sequel. For the record, let us nevertheless mention that one of the next steps of our research project is to understand how supercuspidal representations deform.
\end{rem}

\section{A non-supercuspidal semisimple Langlands correspondence}
\label{section-langlands}
Following Langlands philosophy, (packets of) isomorphism classes of irreducible smooth representations of a reductive $p$-adic group should correspond to equivalence classes of Langlands parameters, which roughly correspond to certain (packets of)  Galois representations. For our group $G = \U(1, 1)(\bQ_{p^2}/\bQ_p)$, this statement has been clarified and proven by Kozio\l \ in \cite[Section 6]{Koziol}. The goal of this section is to explain what happens to packets of non-supercuspidal representations under such a correspondence.

\subsection{Galois representations and dual groups associated with $G$} \label{SectionDefinitionDualGroupsG}
To make precise the conditions put on the Galois representations that should appear in the $p$-modular Langlands correspondence for $G$, we need to define some algebraic groups that are naturally attached to (the algebraic group $\U(1,1)$ that defines) $G$ in this setting. The first one is the $L$-group, or Langlands group, which already appears in the classical (complex) correspondence and is defined as follows. (Recall that $\Frob$ is the Frobenius automorphism fixed in Section \ref{section-intro}.)
\begin{Deff}
\label{DefinitionLgroup}
The {\it $L$-group associated to $G$} is defined as $~^LG := \rm{GL}_{2}(\FFF_{p}) \rtimes \Gamma_{\Q_{p}}$, where the action of $\Gamma_{\Q_{p}}$ on $\GL_{2}(\FFF_{p})$ is given by the following formulae (for $g \in {\rm{GL}}_{2}(\FFF_{p})$ and $h \in \Gamma_{\Q_{p^{2}}}$):
\begin{eqnarray}
\label{ActionGaloisLgroupFrob}
 \Frob g \Frob^{-1} &=& \begin{psmallmatrix}{0} &{1}\\{-1} & {0}\end{psmallmatrix} (g^t)^{-1} \begin{psmallmatrix}{0} &{1}\\{-1} & {0}\end{psmallmatrix}^{-1} = \det(g)^{-1}g \ , \\
\label{ActionGaloisLgroupInertia} \ hg h^{-1}&=& g \ . 
\end{eqnarray}
\end{Deff}
For further use and generalisation, we note here that the group $\GL_{2}(\FFF_{p})$ involved in Definition \ref{DefinitionLgroup} is actually the group of $\FFF_{p}$-points of the usual {\it dual group} of $\U(1,1)$. As the latter splits over the quadratic unramified extension $\Q_{p^{2}}$ of $\Q_{p}$, the $\FFF_{p}$-points of its dual group naturally identify with the $\FFF_{p}$-points of the dual group of the split form, i.e. of ${\rm{GL}}_{2}$. 

The second group we are interested in is the $C$-group, as defined in \cite[Definition 5.3.2]{BuzzardGee} by Buzzard and Gee. As explained in the introduction of \cite{BuzzardGee}, their initial motivation was to define an alternative to the $L$-group that makes it possible to state reasonable generalisations of Clozel's algebraicity conjectures for groups other than $\GL_{n}$. Since this alternative appears in Kozio\l--Morra's work as a convenient tool to deform certain families of Galois parameters \cite{KM}, we need to understand how it relates to the classical $L$-group ${}^LG$ used by Kozio\l \ to state the semisimple correspondence \cite[Definition 6.20]{Koziol}. Note that defining the $C$-group of $G$ requires us to assume $p \neq 2$, though the previously mentioned semisimple correspondence holds for an arbitrary prime $p$.

Thus assume for now that $p$ is odd. The {\it $C$-group associated to $G$} is the $L$-group of some $\GG_{m}$-extension of the algebraic group $\U(1,1)$ that is characterised by \cite[Proposition 5.3.1]{BuzzardGee}. At this point in the paper, we only need  the following concrete description of its $\FFF_{p}$-points. If we let $^CG$ denote this group, then we have (as in \cite[Section 8.3]{BuzzardGee} or \cite[Appendix A]{Koziol}):
\begin{equation}
\label{DefinitionCgroup}
\displaystyle ^CG = \left(\left( \rm{GL}_{2}(\FFF_{p}) \times \FFF^{\times}_{p}\right) / \langle (-I_{2}, -1) \rangle\right) \rtimes \Gamma_{\Q_{p}} \simeq  \left( \rm{GL}_{2}(\FFF_{p}) \times \FFF^{\times}_{p}\right) \rtimes \Gamma_{\Q_{p}} \ , 
\end{equation}
where $\Gamma_{\Q_{p}}$ acts trivially on the $\FFF_{p}^{\times}$-factors and by formulae \eqref{ActionGaloisLgroupFrob} and \eqref{ActionGaloisLgroupInertia}  on the $\GL_{2}(\FFF_{p})$-factors (on both sides of the isomorphism). Note that the latter isomorphism in \eqref{DefinitionCgroup} is given by 
\begin{equation}
\label{eqn-Cgrpiso}
(g, a) \mapsto \left(\begin{pmatrix} a & 0\\0 & a \end{pmatrix}g, a^{2} \right).
\end{equation}

\subsection{From Langlands parameters to $C$-parameters}\label{section-Cparams-def}
We now recall classical definitions related to Langlands parameters involving the $L$-group of $G$. Later, we will adjust our definitions to work with parameters defined using the $C$-group of $G$, to connect more clearly to the setting developed in \cite{KM}. To ease notation, we let $\hat{G} := \rm{GL}_{2}(\FFF_{p})$ denote the left-hand factor in the definition of $~^LG$.
\begin{Deff}
\label{DefLparameter}
A \emph{Langlands parameter} is a group homomorphism $\phi: \AbGal_{\bQ_p} \to ~^LG = \rm{GL}_{2}(\FFF_{p}) \rtimes \Gamma_{\Q_{p}}$  such that the composition of $\phi$ with the canonical projection $~^LG \to \AbGal_{\bQ_p}$ is the identity map.
We say that two Langlands parameters are \emph{equivalent} if they are conjugate by an element of $\hat{G}$.
\end{Deff}   
Let $\phi : \Gamma_{\Q_{p}} \to ~^LG$ be a Langlands parameter. According to \eqref{ActionGaloisLgroupInertia}, $\Gamma_{\Q_{p^{2}}}$ acts trivially on $\hat{G}$, hence the restriction of $\phi$ to $\Gamma_{\QQ_{p^{2}}}$ is of the form $h \mapsto \phi_{0}(h)h$ for some group homomorphism $\phi_{0} : \Gamma_{\Q_{p^{2}}} \to \hat{G}$.
\begin{Deff}
\label{DefinitionRepGalForLparam}
Given a Langlands parameter $\phi$, the group homomorphism $\phi_{0} : \Gamma_{\Q_{p^{2}}} \to \rm{GL}_{2}(\FFF_{p})$ defined above is a two-dimensional $p$-modular representation of $\Gamma_{\QQ_{p^{2}}}$ called the {\it Galois representation associated to $\phi$}.
\end{Deff}
Note that, by construction, a Langlands parameter $\phi$ is completely determined by the image $\phi(\Frob)$ of the Frobenius element and by its Galois representation $\phi_{0}$. But, unlike in the classical complex setting, the Galois representation $\phi_{0}$ is always reducible \cite[Proposition 6.13]{Koziol}. This implies that its image lies in a Borel sugroup of $\rm{GL}_{2}(\FFF_{p})$, which is (up to conjugacy, which does not change the isomorphism class of $\phi_{0}$ as a representation of $\Gamma_{\Q_{p^{2}}}$) the subgroup $P$ of its upper-triangular matrices. Equivalently, this means that $\phi_{0}$ is an extension of the form $1 \rightarrow \chi_{1} \rightarrow \star \rightarrow \chi_{2} \rightarrow 1$, where $\chi_{1}, \chi_{2} : \Gamma_{\QQ_{p^{2}}} \to \FFF_{p}^{\times}$ are continuous characters. Since non-supercuspidal representations come by definition from irreducible smooth representations of $T$ (as subquotients of well-chosen parabolically induced representations, see Definition \ref{DefnonspcuspG}), they should correspond, under the $p$-modular Langlands correspondence for $G$, to Galois representations coming from the corresponding Levi factor of $P$, i.e. to Galois representations that decompose as $\chi_{1} \oplus \chi_{2}$ for some continuous characters $\chi_{1}, \chi_{2} : \Gamma_{\QQ_{p^{2}}} \to \FFF_{p}^{\times}$. Letting $P := ZN$ be the Levi decomposition of $P$, with $N$ being the unipotent radical of $P$, this observation translates on the level of Langlands parameters by requiring that the corresponding Langlands parameter $\phi$ factors through the natural $\Gamma_{\Q_{p}}$-equivariant embedding $~^LZ \hookrightarrow ~^LG$ induced by the canonical embedding $Z \hookrightarrow G$. We will say that such Langlands parameters are {\it non-supercuspidal}.

To give an explicit description of such parameters (as in \cite[Proposition 6.17]{Koziol}), we need to introduce the following notation. For any $\lambda \in \FFF_{p}^{\times}$, we let $\mu_{2, \lam}:\AbGal_{\bQ_{p^2}} \to \overline{\bF}_p^\times$ be the group homomorphism that is trivial on $\Inertia_{\bQ_{p^2}}$ and satisfies $\mu_{2, \lam}(\Frob^2) = \lam$. {Now fix an element $\varpi_2 \in \overline\bQ_p$ such that $\varpi_2^{p^2 - 1} = p$, and let $\varpi_1 = \varpi_2^{p + 1}$. For $n \in \{1, 2\}$, we let $\omega_{n} : \Inertia_{\Q_{p}} \to \FFF_{p}^{\times}$ be the character given by $\omega_n(h) = \overline{(\frac{h\cdot \varpi_n}{\varpi_n})}$}. According to \cite[Lemma 2.5]{Breuil}, $\omega_{n}$ can be extended to a smooth character $\omega_{n} : \Gamma_{\Q_{p^{n}}} \to \FFF_{p}^{\times}$ that maps $\Frob^{n}$ to $1$. It is easy to check \cite[Corollary 6.3]{Koziol} that smooth characters of $\Gamma_{\Q_{p^{2}}} \to \FFF_{p}^{\times}$ are of the form $\mu_{2, \lam}\omega_2^r$ for some pair $(r, \lam) \in \bZ \times \overline{\bF}_p^\times$ with $0 \leq r \leq p^{2} - 2$.

\begin{Prop}
\label{ClassifNonSupercuspLParameters}
A Langlands parameter $\phi :  \AbGal_{\bQ_p} \to ~^LG$ is non-supercuspidal if, and only if, there exists a pair $(r, \lambda) \in \bZ \times \FFF_{p}^{\times}$ with $0 \leq r \leq p^{2} - 2$ such that $\phi$ is equivalent to the Langlands parameter $\psi_{r,\lambda,}$ defined by 
\begin{equation*}
\displaystyle \psi_{r, \lam}(\Frob)= \begin{pmatrix} 1 & 0\\0 & \lam\end{pmatrix} \Frob \text{ and }  \psi_{r, \lam}(h) = \begin{pmatrix} \mu_{2, \lam^{-1}}\omega_2^{r}(h) & 0\\ 0 & \mu_{2, \lam}\omega_2^{-pr}(h)\end{pmatrix} h
\end{equation*}
for any $h \in \AbGal_{\bQ_{p^2}}$. Moreover, if $(r, \lam) \in \bZ \times \overline{\bF}_p^\times$ and $(r', \lambda') \in \bZ \times \FFF_{p}^{\times}$ are such that $0 \leq r, r' \leq p^{2} - 2$, then $\psi_{r,\lambda}$ and $\psi_{r',\lambda'}$ are equivalent Langlands parameters if, and only if, we have $(r', \lambda') = (r, \lambda)$ or $(r', \lambda') = (-pr + m(p^2 - 1), \lam^{-1})$ for some $m \in \bZ$.
\end{Prop}
\begin{proof}
The characterisation of non-supercuspidal parameters comes from \cite[Proposition 6.17]{Koziol}. The characterisation of equivalence classes of such parameters is done in \cite[Lemma 6.18]{Koziol}.
\end{proof}

As mentioned earlier, we would like to connect these parameters to $C$-groups: this is why we now define the notion of $C$-parameters (also called $^CG$-valued Langlands parameters), and explain how they relate to the (non-supercuspidal) Langlands parameters defined above. In the next definition, we write $^C\hat{G}$ for the left-hand factor {$( \rm{GL}_{2}(\FFF_{p}) \times \FFF^{\times}_{p}) / \langle(-I_{2}, -1) \rangle$} appearing in the definition of $^CG$ (see \eqref{DefinitionCgroup}).
\begin{Deff}
\label{DefCparameter}
A {\it $C$-parameter} is a group homomorphism $\phi : \AbGal_{\bQ_p} \to ~^CG$ such that the composition of $\phi$ with the canonical projection $^CG \to \AbGal_{\bQ_p}$ is the identity map.
We say that two $C$-parameters are {\it equivalent} if they are conjugate by an element of $\dualC$.
\end{Deff}
Under \eqref{eqn-Cgrpiso}, we can (and will) consider $C$-parameters to take values in $(\GL_2(\overline{\bF}_p) \times \overline\bF_p^\times) \rtimes \AbGal_{\bQ_p}$; hence we will also use isomorphism \eqref{eqn-Cgrpiso} to identify $^C\hat{G}$ with $\GL_2(\overline{\bF}_p) \times \overline\bF_p^\times$. This point of view allows for a more straightforward connection with Langlands parameters, as the latter take values in $\GL_{2}(\FFF_{p}) \rtimes \Gamma_{\Q_{p}}$. Similarly to what was done in Proposition \ref{ClassifNonSupercuspLParameters} for Langlands parameters, we define $C$-parameters $\widetilde{\psi}_{r,\lambda}$, indexed by pairs $(r, \lambda) \in  \bZ \times \FFF_{p}^{\times}$ with $0 \leq r \leq p^{2}-2$, by setting
\begin{eqnarray*} 
 \nonumber \widetilde{\psi}_{r, \lam}(\Frob) &=& \left(\begin{pmatrix} 1 & 0\\0 & \lam\end{pmatrix}, 1\right) \Frob \text{ and}\\
\nonumber \widetilde{\psi}_{r, \lam}(h) &=& \left(\begin{pmatrix} \mu_{2, \lam^{-1}}\omega_2^{r}(h) & 0\\ 0 & \mu_{2, \lam}\omega_2^{-pr - (p + 1)}(h)\end{pmatrix}\omega_1(h), \omega_1(h)\right) h
\end{eqnarray*} 
for any $h \in \Gamma_{\Q_{p^{2}}}$. 

We observe that, by construction, there is a one-to-one correspondence between these $C$-parameters and non-supercuspidal Langlands parameters as given in Proposition \ref{ClassifNonSupercuspLParameters}. In particular, equivalence classes of $C$-parameters are characterised in the same way as equivalence classes of Langlands parameters are in the last part of Proposition \ref{ClassifNonSupercuspLParameters}.

\subsection{A Langlands correspondence for non-supercuspidal representations}
Historically, the first group for which the formulation of a (classical or $p$-modular) local Langlands correspondence involves actual packets (i.e. is not a one-to-one correspondence) is the special linear group $\SL_{2}$. Thus before we discuss $L$-packets for $G$ in the present setting, we briefly review the logic that goes into making analogous definitions for $\SL_{2}$. In the classical setting (see the introduction of \cite{LL79}), packets for $\SL_2$ can be seen (on the automorphic side) as an incarnation of the action of $\GL_{2}$ by conjugation on smooth representations of $\SL_{2}$, as most packets correspond to $\GL_{2}$-orbits of irreducible representations. Each of these so-called $L$-packets corresponds to an equivalence class of group homomorphisms $\Gamma_{\QQ_{p}} \to \PGL_2(\CC)$, which in turn can be thought of as the set of Galois representations lifting a fixed projective Galois representation.

In the $p$-modular setting, we can follow the same philosophy to define a local Langlands correspondence for supercuspidal representations of $\SL_{2}(\QQ_{p})$ \cite[Section 4.2]{Abdellatif12}: in this case, $L$-packets consist of $\GL_{2}(\QQ_{p})$-orbits of irreducible supercuspidal representations of $\SL_{2}(\QQ_{p})$ over $\FFF_{p}$, but these packets can also be defined as Jordan-H$\ddot{\text{o}}$lder factors of a given irreducible supercuspidal representation of $\GL_{2}(\QQ_{p})$. Note that, unlike in the classical case, there is no multiplicity-one result for $p$-modular representations (compare \cite[Lemma 2.6]{LL79} with \cite[Th\'eor\`eme 4.12 (2)]{Abdellatif12} for $r = (p-1)/2$). 

The non-supercuspidal case for $\SL_2(\bQ_p)$ is even trickier than the supercuspidal case. First note that, as twisting a representation by {a} character of $\GL_2(\bQ_p)$ does not change its restriction to $\SL_2(\bQ_p)$, non-isomorphic representations of $\GL_{2}(\QQ_{p})$ can be isomorphic as representations of $\SL_{2}(\QQ_{p})$. This can actually occur even for pairs of representations that remain non-isomorphic after twisting by any character of $\GL_{2}(\QQ_{p})$, see \cite[Th\'eor\`emes 2.16 and 4.12]{Abdellatif12}. Moreover, \cite[Proposition 2.8]{Abdellatif12} shows that non-supercuspidal representations are fixed (up to isomorphism) under the action of $\GL_{2}(\QQ_{p})$ by conjugation, which prevents us from defining $L$-packets simply as $\GL_{2}$-orbits of irreducible smooth representations. Indeed, doing so would imply that any natural map from irreducible representations of $\SL_2(\bQ_p)$ to the set of equivalence classes of Langlands parameters for $\SL_2(\bQ_p)$ (i.e. of projective Galois representations in this context) would map distinct $L$-packets to the same projective representation, which does not make sense for a correspondence. This motivates the introduction of (reducible) semisimple representations in the non-supercuspidal setting \cite[Definition 4.13]{Abdellatif12}: their factors correspond to irreducible representations that should match with the same projective representation, and their construction ensures that they are stable under $\GL_{2}$-conjugation, hence that the original concept of an $L$-packet is in some way preserved. Note that \cite[Th\'eor\`eme 3.18]{Abdellatif12} relates these semisimple representations with semisimplifications of well-known smooth representations of $\SL_{2}(\QQ_{p})$.

Following the same philosophy, we are led to make the following definitions for $G$.
\begin{Deff}
\label{DefinitionSemiSimpleNonSupercusp}
{Given $(r, \lam) \in \bZ \times \FFF_{p}^{\times}$ with $0 \leq r \leq p^{2} -2$, we write $\pi(r, \lambda)$ for the semisimplification of $\Ind_{B}^{G}(\mu_{\lambda^{-1}}\omega^{-pr})$.}
\end{Deff}
For instance, we have, for $0 \leq r \leq p-1$ and $\lambda \in \FFF_{p}^{\times}$:
\begin{eqnarray*}
\nonumber \pi(r, \lambda)= 
\begin{cases}
\Ind_B^G(\mu_{\lambda^{-1}} \omega^{-pr}) & \text{ if } (r,\lambda) \neq (0,1), (p-1,1), \\ 
 (\omega^p \circ \det ) \oplus  ((\omega^p \circ \det) \otimes \St_G)   & \text{ if } (r,\lambda) = (p-1,1), \\
\1_{G} \oplus \St_G  & \text{ if } (r,\lambda) = (0,1).
\end{cases}
\end{eqnarray*}

\begin{rem}
Although Definition \ref{DefinitionSemiSimpleNonSupercusp} is uniform for all pairs $(r, \lam)$, the representation $\pi(0,1)$ actually naturally arises as the semisimplification of a non-trivial extension of $\1_G$ by $\St_G$, not of a parabolically induced representation. Though it has no impact in the definition of a {\it semisimple} Langlands correspondence, it must be kept in mind for future work. For more information on this phenomenon, {the reader may} refer to \cite[Th\'eor\`emes 3.16 and 3.18]{Abdellatif12}, where the corresponding objects and phenomena for $\GL_{2}$ and $\SL_{2}$ are introduced.
\end{rem}

Note that assuming $0 \leq r \leq p-1$ is not restrictive in the context of Langlands correspondences, since Proposition \ref{ClassifNonSupercuspLParameters} shows that an equivalence class of Langlands parameters (or, equivalently, of $C$-parameters) always contains a parameter whose index $r$ satisfies this condition. Also note that, similarly to what happens for $\SL_{2}$ and following \cite[Proposition 5.10]{Koziol}, each of these representations is fixed (up to isomorphism) under the action of 
\[\displaystyle \GU(1,1)(\QQ_{p^{2}}/\QQ_{p}) := \left\{g \in \GL_{2}(\QQ_{p^{2}}) \ \vert \   g^{*} \begin{pmatrix} 0 & 1  \\ 1 & 0  \end{pmatrix} g = \kappa  \begin{pmatrix} 0 & 1  \\ 1 & 0  \end{pmatrix} \text{for some } \kappa \in \QQ_{p}^{\times} \right\} \ . \]
According to Proposition \ref{ClassifNonSupercuspLParameters}, distinct pairs $(r, \lam)$ and $(r', \lam')$ may correspond to equivalent $C$-parameters, so we are led to  define the following semisimple representations of $G$ to obtain suitable representatives of $L$-packets.
\begin{Deff}
For $(r, \lambda) \in \bZ \times \FFF_{p}^{\times}$ with $0 \leq r \leq p-1$, we define $\Pi(r, \lambda)$ as the following semisimple representation of $G$:
\begin{equation*}
\Pi(r, \lam) := \pi(r, \lambda) \oplus (\omega^{r+1} \circ \det) \otimes \pi(p-1-r, \lam^{-1}) \ .
\end{equation*} 
For any integer $k$, we also set $\Pi(r, \lam, k) := (\omega^{k} \circ \det) \otimes \Pi(r, \lambda)$.
\end{Deff}
Note that $\Pi(r, \lambda, k)$ is an $L$-packet for $G$ in the sense of \cite[Definition 5.9]{Koziol}, since it is by construction fixed under the action of $\GU(1,1)(\QQ_{p^{2}}/\QQ_{p})$. We can now state the non-supercuspidal part of the $p$-modular semisimple Langlands correspondence for $G$ \cite[Definition A.5]{Koziol}.
\begin{Def}
\label{SemiSimpleLanglands}
The non-supercuspidal part of the semisimple $p$-modular correspondence for $G$ is the following matching between equivalence classes of $C$-parameters $\widetilde{\psi}_{r',\lambda'}$ and isomorphism classes of $L$-packets $\Pi(r,\lambda,k)$: for any $0 \leq r \leq p-1, \ \lambda \in \overline{\F}_p^{\times}$ and $0 \leq k < p+1$,
\[ \widetilde{\psi}_{(r-1)+(1-p)k,\lambda}  \longleftrightarrow \Lpacket(r, \lam, k) \ . \]
\end{Def}

\section{Deforming non-supercuspidal representations of $G$}\label{section-nonsc-deformations}
Our purpose is to study how the $p$-modular Langlands correspondence defined above behaves when it is functorially lifted to characteristic $0$, i.e. to $p$-adic representations/parameters. To define precisely what we mean by functorial lifting requires the use of deformation theory. This section gathers what is known about deformations of non-supercuspidal representations of $G$; its Galois counterpart, relative to the deformation of Galois parameters, is postponed to Section \ref{SectionDeformationParameters}.

From now on, we let $E/\QQ_{p}$ be a (large enough) finite extension of fields. We let $\OO$ be the ring of integers of $E$ and $k = \OO / \varpi\OO$ be its residue field, where $\varpi$ denotes a fixed uniformiser of $\OO$.
 We write $\Art(\Oe)$ (respectively: $\mathrm{Noe}(\OO)$; $\mathrm{Pro}(\OO)$) for the category of local Artinian (respectively: local complete Noetherian; local profinite) $\Oe$-algebras 
 {$A$ such that the structural morphism from $\OO$ to $A$ is local and induces an isomorphism between $k$ and the residue field of $A$;
 the morphisms are local $\OO$-algebra morphisms (respectively: local $\OO$-algebra morphisms; continuous local $\OO$-algebra morphisms).}

Note that $\Art(\OO)$ is the full subcategory of Artinian rings {(in $\mathrm{Pro}(\OO)$ and)} in $\mathrm{Noe}(\OO)$  and that $\mathrm{Noe}(\OO)$ is the full subcategory of Noetherian rings in $\mathrm{Pro}(\OO)$.

We write $H$ for the $\QQ_{p}$-points of a connected reductive group defined over $\QQ_{p}$. (In the sequel, $H$ will mainly be either $G$ or $T$, as defined above.)
\begin{Deff}
Let $\overline\pi$ be a representation of $H$ over $\rese$.
A \textit{lift} of $\overline\pi$ to $A \in \Art(\Oe)$ is a pair $(\pi, \phi)$ where $\pi$ is a smooth $A[H]$-module that is free over $A$, and $\phi: \pi \to \overline\pi$ is an $A[H]$-linear surjection that induces an $A[H]$-linear isomorphism 
\begin{equation*}
\pi \otimes_A k \to \overline\pi \ .
\end{equation*}
A {\it morphism of lifts} $(\pi, \phi) \to (\pi', \phi')$ is an $A[H]$-linear morphism $i: \pi \to \pi'$ such that $\phi = \phi' \circ i$. 
\end{Deff} 
For any $A \in \Art(\OO)$, we let $\Deform_{\overline\pi}(A)$ be the set of isomorphism classes of lifts of $\overline\pi$ to $A$. Then any smooth $k$-representation $\overline{\pi}$ of $H$ defines a functor 
\[ \Deform_{\overline\pi}: \Art(\Oe) \to \text{Set} \ . \] 
Considering the strong connection between non-supercuspidal representations and parabolic induction functors, it seems natural to wonder how much these deformation functors are compatible with the parabolic induction functor $\Ind_{B}^{G} : \mathrm{Mod}_{T}^{\infty} \to \mathrm{Mod}_{G}^{\infty}$. In a series of two papers \cite{HSS1, HSS2}, Hauseux--Schmidt--Sorensen addressed this question: we now recall the results we need from these papers.

\subsection{Deforming parabolically induced representations}\label{section-deform-ind}
Let $\overline\chi: T \to k^\times$ be a smooth character and let $\chi$ be a lift of $\overline\chi$ to $A \in \Art(\Oe)$. By \cite[Lemma 2.1]{HSS1}, the parabolically induced representation $\Ind_B^G\chi$ is a free $A$-module, and the natural surjection $A \to k$  induces a $k[G]$-linear isomorphism
\begin{equation*}
(\Ind_B^G \chi) \otimes_A k \simeq \Ind_B^G\overline\chi.
\end{equation*}
In other words, $\Ind_B^G\chi$ is a lift of $\Ind_B^G\overline\chi$ over $A$. Furthermore, any morphism of lifts $(\chi_1, \phi_{1}) \to (\chi_2, \phi_{2})$ of $\overline\chi$ over the same ring $A$ induces a morphism $(\Ind_B^G \chi_1, \Ind_{B}^{G}(\phi_{1})) \to (\Ind_B^G\chi_2, \Ind_{B}^{G}(\phi_{2}))$ of lifts of $\Ind_{B}^G(\overline\chi)$ over $A$. In \cite[Section 2.5]{HSS1}, Hauseux--Schmidt--Sorensen show that we actually have a morphism of functors
\begin{equation*}
\Ind_B^G: \Deform_{\overline\chi} \to \Deform_{\Ind_B^G\overline\chi} \ . 
\end{equation*}
They also give sufficient conditions for this morphism to be an isomorphism, which lead to the following statement for the group $G$.
\begin{Thm}
\label{ThmDeformationIndBbar} 
Let $\overline\chi = \mu_\lam\omega^r$ be a smooth $k$-character of $T$. Assume that $\lam \neq \pm 1$ or that $r - 1$ is not divisible by $p - 1$. Then 
\begin{equation*}
\Ind_B^G: \Deform_{\overline\chi} \to \Deform_{\Ind_B^G\overline\chi}
\end{equation*}
is an isomorphism. 
\end{Thm}
Note that this statement requires us to assume that $E/\QQ_{p}$ is large enough, to ensure that $\overline{\chi}$ actually takes values in the finite field $k$, and not only in $\overline{k}$.
\begin{proof}
Let $\Delta = \{\al\}$ be the set of positive roots for $G$ with respect to $(B^{-}, T)$, and write $s_\al$ for the corresponding reflection in the Weyl group. 
Following \cite[Corollary 4.18]{HSS1} for $F = \QQ_{p}$, it suffices to check that 
\begin{equation}\label{eqn-HSScondition}
\overline\chi \neq \overline\chi^\al \otimes (\omega^{-1}\circ \al).
\end{equation}
Since $\overline\chi^\al(x) = \overline\chi(c(x)^{-1})$ for all $x \in \bQ_{p^2}$, checking (\ref{eqn-HSScondition}) is equivalent to find some  $x \in \bQ_{p^2}$ such that
\begin{equation*}
\overline\chi(xc(x)) \neq (\omega^{-1} \circ \al)(x), \text{ i.e. such that } \overline\chi(N(x)) \neq (\omega^{-1} \circ \al)(x) \ , 
\end{equation*}
where $N: \bQ_{p^2} \to \bQ_p$ denotes the norm map.
We now want to explicitly describe $\al$: to do this, let us recall that we fixed an element $\varepsilon \in \QQ_{p^{2}}$ such that $\{ 1, \varep\}$ is a basis for $\bQ_{p^2}$ over $\bQ_p$ and $c(\varep) = -\varep$. If $\frakg$ denotes the Lie algebra of $G$, then a basis for the Lie algebra of $G$ over $\bQ_p$ is given by 
\begin{equation*}\left\{
\begin{pmatrix}
1 & 0\\
0 & -1
\end{pmatrix},
\begin{pmatrix}
\varep & 0\\
0 & \varep
\end{pmatrix},
\begin{pmatrix}
0 & \varep\\
0 & 0
\end{pmatrix},
\begin{pmatrix}
0 & 0\\
\varep & 0
\end{pmatrix}
\right\} \ .
\end{equation*}
As any element $t(x) = \begin{psmallmatrix}
x & 0\\
0 & c(x)^{-1}
\end{psmallmatrix}$
of $T$ acts on $\begin{psmallmatrix} 0 & 0\\ \varepsilon & 0 \end{psmallmatrix}$ as multiplication by
$x^{-1}c(x)^{-1} = N(x)^{-1}$, we see that $\al$ is given by the inverse of the norm map, hence checking condition (\ref{eqn-HSScondition}) now boils down to finding some $x \in N(\bQ_{p^2}^\times)$ such that $\overline\chi(x) \neq \omega(x)$. But we have, for any $u \in \ZZ_{p}^{\times}$ and any integer $n$,
\[ \overline\chi(up^{2n}) = \lam^{2n}\overline u^r \ . \]
If $\lam \neq \pm 1$, then $\overline\chi(p^2) \neq \omega(p^2)$. If $r - 1$ is not divisible by $p - 1$, then there exists some unit $u \in \bZ_p^\times$ such that $\omega(u) \neq \overline\chi(u)$. As $N(\bQ_{p^2}^\times) = \langle p^2 \rangle \times \bZ_p^\times$, the result follows.
\end{proof}
Further note that under the assumptions of Theorem \ref{ThmDeformationIndBbar}, the deformations of $\overline{\chi}$ (hence of $\Ind_{B}^{G}(\overline{\chi})$) are well-understood: according to \cite[Proposition 4.17]{HSS1}, $\Deform_{\Ind_{B}^{G}(\chi)}(A)$ is in natural bijection with $\mathrm{Hom}_{\mathrm{Pro(\OO)}}(\Lambda, A)$ for any $A \in \Art(\OO)$, the converse map being given by $\psi \mapsto \Ind_{B}^{G}(\psi \circ \chi^{\mathrm{univ}})$, and this bijection is functorial in $A \in \Art(\OO)$. In this statement, $\Lambda$ denotes the Iwasawa algebra associated to the torus $T$ (see \cite[Section 19.7]{Schneider} for a precise definition) and $\chi^{\mathrm{univ}} : T \to \Lambda$ is the so-called {\it universal deformation} of $\overline{\chi}$ (see \cite[Proposition 4.17]{HSS1} for the explicit formula defining $\chi^{\mathrm{univ}}$). To completely understand deformations of parabolically induced representations, we now have to answer the following open question.
\begin{What}
What are the deformations of $\Ind_{B}^{G}(\mu_{\lambda}\omega^{r})$ when $\lambda = \pm 1$, or when $p-1$ divides $r-1$?
\end{What}
Note that a direct comparison with Corollary \ref{CorPrincipalSeries} shows that the major part of the representations covered by this question are {\it irreducible} representations of $G$, hence of greatest interest in the Langlands program.

\subsection{Deformations of special series representations}

Given $A \in \Art(\Oe)$ and a smooth character $\sigma: T \to A^\times$ that extends to a smooth character of $G$, we can proceed as in the $p$-modular setting (see \cite[Section 5.3.2]{Abdellatif}) to show that $\sigma$ is a subrepresentation of $\Ind_{B}^{G}(\sigma)$. We can hence define the Steinberg representation as the quotient representation
\begin{equation*}
\St_{B}^{G}(\sigma) := \Ind_{B}^{G}(\sigma)/\sigma.
\end{equation*}
Note that, if $A = k$ and $\sigma$ is a smooth $k$-character of $T$ that extends to a smooth $k$-character of $G$, then $\St_{B}^{G}(\sigma) \simeq \St_{G} \otimes \, \sigma$, where $\St_{G}$ denotes the Steinberg representation we defined in Section \ref{section-steinberg}. Following \cite[Section 9]{HSS2}, we obtain, for any smooth character $\chi : G \to k^{\times}$, a natural transformation
\begin{equation}
\label{DeformationSteinberg}
\St_{B}^{G}: \Deform_{\overline\chi} \to \Deform_{\St_{G} \otimes \overline{\chi}} \ .
\end{equation}
\begin{Thm}
The natural transformation \eqref{DeformationSteinberg} is an isomorphism of functors.
\end{Thm}
\begin{proof}
This is a straightforward application of \cite[Proposition 10]{HSS2}, as one-dimensional representations are obviously admissible.
\end{proof}
In other words, understanding deformations of representations in the special series amounts to understanding deformations of smooth $k$-characters of $G$, and the latter are once again well-understood by \cite[Proposition 4.17]{HSS1}. Following \cite[Corollary 15]{HSS2}, and using the same notation as {in Section \ref{section-deform-ind}}, we obtain a similar statement to the one we get for parabolically induced representations: for any smooth character $\overline{\chi} : G \to k^{\times}$ and any $A \in \mathrm{Noeth}(\OO)$, we have a natural bijection between $\Deform_{\St_{G} \otimes \overline{\chi}}(A)$ and $\mathrm{Hom}_{\mathrm{Pro(\OO)}}(\Lambda, A)$, whose converse map is given by $\psi \mapsto \St_{B}^{G}(\psi \circ \chi^{\mathrm{univ}})$, and this bijection is moreover functorial in $A \in \Art(\OO)$.

\section{Deforming Galois parameters} 
\label{SectionDeformationParameters} 
In this section, we summarise what is known about deformations of the parameters defined in Section~\ref{section-Cparams-def} and explain what questions we aim to solve. One can view this section as the Galois counterpart of Section \ref{section-nonsc-deformations} with the difference that, as mentioned in the introduction, a priori we make no {specific} assumption on the Galois parameters here
(while Section \ref{section-nonsc-deformations} only holds for non-supercuspidal representations). We follow \cite{KM} for most of the section.
 
 We {start by} introducing $C$-parameters in characteristic zero as well as the notion of an inertial type. We also introduce genericity for $C$-parameters as we will need to impose some genericity conditions on the Galois parameters we want to deform. We recall some notions from deformation theory, specialised to our setting. We then show how one transfers from $C$-parameters to genuine $p$-adic Galois representations. It turns out that we can study the deformations of these representations by relating them to Kisin modules. We end with demonstrating to what extent we can determine the deformations of our Galois parameters explicitly using Kisin modules.

We keep the same notation as above and we furthermore fix an embedding $\sigma_{0} : \Q_{p^{2}} \hookrightarrow E$, as well as a $(p^{2}-1)^{\text{th}}$ root $\pi$ of $-p$ in $E$ (which exists as $E/\bQ_p$ is assumed to be large enough). We let $L = \Q_{p^{2}}(\pi)$ and write $\widetilde{\omega}_{\pi} : \Gal(L/\Q_{p^{2}}) \to \Z_{p^{2}}^{\times}$ for the character given by $\widetilde\omega_\pi(g) = \frac{g(\pi)}{\pi}$. We also set $\widetilde{\omega}_{2} := \sigma_{0} \circ \widetilde{\omega}_{\pi} : \Gal(L/\Q_{p^{2}}) \to \OO^{\times}$. {Note that $\widetilde\omega_\pi$ and $\widetilde\omega_2$ do not depend on the choice of $\pi$.}

\subsection{Definition of $C$-parameters in characteristic zero}
To define a suitable notion of $C$-parameters in characteristic $0$, we first have to define the $C$-group in characteristic $0$, which is actually straightforward. Indeed, as mentioned in Section \ref{SectionDefinitionDualGroupsG}, the group $^CG$ arises as the $\FFF_{p}$-points of a certain algebraic group, which we will call $^C\GG$. Thus we may extend our definition to define $^C\GG(R)$ for any topological $\ZZ_{p}$-algebra $R$ as follows (see \cite[Section 2.3]{KM} for further details):
\[ 
^C\hat{\GG}(R) := \left( \GL_{2}(R) \times R^{\times}\right) / \langle (-I_{2}, -1) \rangle \text{ and } 
^C\GG(R) := {}^C\hat{\GG}(R) \rtimes \Gamma_{\Q_{p}},\]
where $\Gamma_{\Q_{p}}$ still acts via formulae \eqref{ActionGaloisLgroupFrob} and \eqref{ActionGaloisLgroupInertia}. 
Note that we will continue to write $^C\hat{G}$ (respectively $^CG$) instead of $^C\hat{\GG}(\FFF_{p})$ (respectively $^C{\GG}(\FFF_{p})$). Also note that {the} isomorphism \eqref{eqn-Cgrpiso} still holds in this setting, so we can (and will) identify $^C\hat{\GG}(R)$ with $\GL_{2}(R) \times R^{\times}$ and $^C\GG(R)$ with $(\GL_{2}(R) \times R^{\times}) \rtimes \Gamma_{\Q_{p}}$. {The topology on $^C\GG(R)$ is the one inherited from $R$ when $^C\hat{\GG}$ is considered as algebraic group. Note that this topology coincides with the usual topology on $ \left( \GL_{2}(R) \times R^{\times}\right) / \langle (-I_{2}, -1) \rangle \rtimes \AbGal_{\bQ_p}$ for the rings considered in the remainder of the paper.} 
We can now define (equivalence classes of) parameters as follows, where $R$ denotes a topological $\Z_{p}$-algebra.
\begin{Deff}
\label{DefCparameterCar0}
An {\it ($R$-valued) $C$-parameter} is a continous homomorphism $\rho : \Gamma_{\Qp} \to {}^C\GG(R)$ such that the composition of $\rho$ with the canonical projection ${}^C\GG(R) \to \Gamma_{\Q_{p}}$ is the identity map. We say that two $C$-parameters are {\it equivalent} if they are conjugate by an element of $^C\hat{\GG}(R)$.
\end{Deff}
Note that formula \eqref{ActionGaloisLgroupInertia} ensures that this definition is equivalent to the one given in \cite[Definition 4.1]{KM}. We also note that there are similar definitions for the $R$-points of the $L$-group as well as $R$-valued $L$-parameters, and that the $R$-valued analogue of {the} isomorphism \eqref{eqn-Cgrpiso} establishes a natural connection between $R$-valued {Langlands parameters and $R$-valued} $C$-parameters. 

\begin{rem}
\label{RemConnectionCparChar0GalRep}
As in the $p$-modular setting (see Definition \ref{DefinitionRepGalForLparam}), and following \cite[Section 5.3.2]{KM}, there is a natural bijection between $C$-parameters in characteristic $0$ and genuine Galois representations enriched with extra data. This bijection will be made explicit in Section \ref{SectionParametersTwdGalois}.
\end{rem}

\subsection{Inertial types and generic $C$-parameters}
We now characterise the $C$-parameters that will be the focus of the remainder of the paper. A common way to have some control on Galois objects consists in putting some conditions on their restrictions to the inertia subgroup, and we will proceed in this way to define a suitable notion of genericity, inspired by similar notions for genuine representations of absolute Galois groups. 
We thus introduce the notion of an inertial type, which appears in different useful settings that will be described in this subsection. 

\subsubsection{Genericity for classical inertial types}
\label{sssec:generic_types}
We start by recalling the definition of an inertial type. This definition serves as a guide when defining the suitable objects in our setting, and furthermore, we will define constraints on our deformation problems in terms of these objects.
\begin{Deff}
\label{Def:inertial_type}
An {\it inertial type} is a group homomorphism $\tau : \Inertia_{\Q_{p^2}} \to \mathrm{GL}_2(\mathcal{O})$ that has open kernel and extends to a representation of the full Weil group $W_{\QQ_{p^{2}}}$. 
\end{Deff}
Many useful examples of inertial types come from restrictions to $\Inertia_{\Q_{p^{2}}}$ of representations of the full Galois group $\Gamma_{\QQ_{p^{2}}}$. In particular, we consider the following family of inertial types, which we {will} see are essentially the only parameters necessary for studying non-supercuspidal parameters.
\begin{Deff}
Given any pair $(a,b)$ of integers, we write $\tau_{a,b} : \Inertia_{\Q_{p^{2}}} \to \GL_{2}(\OO)$ for the inertial type given by $\widetilde{\omega}_{2}^{a} \oplus \widetilde{\omega}_{2}^{b}$. When $a \not\equiv b \mod p^{2}-1$, we say that $\tau_{a,b}$ is a {\it principal series inertial type}.
\end{Deff}
Note that these inertial types are (by construction) tamely ramified, which means that they are trivial on the wild inertia subgroup. Moreover, if $\tau : \Inertia_{\QQ_{p^{2}}} \to \GL_{2}(\OO)$ is an inertial type whose kernel contains~$\Inertia_{L}$, hence that factors through $\Inertia_{\QQ_{p^2}} / \Inertia_{L} \simeq \Gal(L/\Q_{p^2})$, then there exists a pair of integers $(a,b)$ such that $\tau \simeq \tau_{a,b}$. {(Recall that $L = \bQ_{p^2}(\pi)$, as defined at the beginning of the section.)} 
\begin{rem}
\label{RemTrickySP}
One should be careful with the terminology of ``principal series", as it is not as transparent as it may seem at first. In the classical setting of complex representations, any object on the Galois side labelled as ``principal series" should indeed be related to non-supercuspidal representations (see for instance \cite[Section A.1.2]{BM02}). 

In the current setting, one might expect that principal series inertial types would only be involved in non-trivial deformations of non-supercuspidal parameters, but such an expectation already fails for $\GL_{2}$. Indeed, \cite[Proposition 6.1.2(iii)]{BM02} gives instances of irreducible $p$-modular representations of $\Gamma_{\QQ_{p}}$ that correspond to supercuspidal representations via {the} $p$-modular local Langlands correspondance for $\GL_{2}(\QQ_{p})$ {and that admit} non-zero deformations indexed by a (generic) principal series inertial type. 
\end{rem}

We can define a genericity criterion for these inertial types. As $\widetilde{\omega}_{2}$ is of order $p^{2}-1$, we can assume that $-a$ and $-b$ are non-negative integers both less than $p^{2} - 1$. Their respective $p$-basis expansions are hence of the following form: $-a = a_{0} + pa_{1}$ and $-b = b_{0} + pb_{1}$ with $0 \leq a_{0}, a_{1}, b_{0}, b_{1} \leq p-1$. For further use, it {is} convenient to set ${\bf a} := (a_{0}, a_{1})$ and ${\bf b} := (b_{0}, b_{1})$.
\begin{Deff} 
\label{DefGenericityInertialTypes}
Let $\tau_{a,b}$ be a principal series inertial type, with $a$ and $b$ chosen as above, and let $n$ be a positive integer. We say that $\tau_{a,b}$ is {\it $n$-generic} if : 
\begin{equation}
\label{GenericityConditionPSType}
\forall \ i \in \{0, 1\}, \ n < \vert a_{i} - b_{i} \vert < p - n \ . 
\end{equation}
When condition \eqref{GenericityConditionPSType} is satisfied, we will also say that the pairs $(-a,-b)$ and $(\bf{a}, \bf{b})$ are {\it $n$-generic}.
\end{Deff}

\subsubsection{Genericity for $^C\hat{G}$-valued inertial types}
The next step towards $C$-parameters consists in defining $^C\hat{G}$-valued inertial types, and what genericity means for them, in a {way that is compatible way with the definitions of Section \ref{sssec:generic_types}.}
For any topological $\ZZ_{p}$-algebra $R$, \cite[Definition 4.1]{KM} defines $R$-valued inertial types using the $L$-group. Since $\GL_{2}(R)$ acts by conjugation on such parameters, we can consider them up to equivalence and use the bijection of \cite[Lemma 9.4.5]{GHS} to come back to ${}^C\GG(R)$-valued homomorphisms. We hence get the following definition of ${}^C\GG(R)$-valued inertial types.
\begin{Deff}
Let $R$ be a topological $\Z_{p}$-algebra. A {\it ${}^C\hat{\GG}(R)$-valued inertial type} is a continuous group homomorphism $\Inertia_{\QQ_{p}} \to {}^C\hat{\GG}(R)$ that extends to an $R$-valued $C$-parameter $\Gamma_{\QQ_{p}} \to {}^C\GG(R)$. {Two ${}^C\hat{\GG}(R)$-valued inertial types $\tau_1$ and $\tau_2$ are said to be \textit{equivalent}, written $\tau_1 \simeq \tau_2$, if they are conjugate by an element of ${}^C\hat{\GG}(R)$.}
\end{Deff}
Following \cite[Definition 4.1.4]{KM} for $f = 1$, we now introduce a {family} of inertial types that is the analogue, in this setting, of {the previously defined family of} principal series inertial types. In particular, {it} will be used to define a convenient notion of genericity for $C$-parameters, and we will see later (in Lemma \ref{LemTrivialWeylEltPSType}) that all representations of $\Inertia_{\Q_{p}}$ coming (via {the} Langlands correspondence {of \cite{Koziol}}) from generic non-supercuspidal representations of $G$ can actually be described as such ${}^C\hat{G}$-valued inertial types. 
\begin{Deff}
\label{DefTypesInertielsCvalued}
Given a pair of integers $(a,b)$ and $w \in W$, where $W$ is the Weyl group of $G$, we define $\tau_{w}(a,b) : \Inertia_{\Q_{p}} \to {}^C\hat{G}$ as follows. If $w = 1$, then we set:  
\begin{equation}
\label{EquaTypeIntertielCvaluedWeylTrivial}
\forall \ h \in \Inertia_{\QQ_{p}}, \ \tau_{1}(a,b)(h) = \left(\begin{pmatrix}
\omega_{2}(h)^{a + 1 + p(1 - b)} & 0\\
0 & \omega_{2}(h)^{b - pa}
\end{pmatrix}, \omega_1(h)\right) \ . 
\end{equation}
If $w = \weyl$, then we set:
\begin{equation}
\label{EquaTypeIntertielCvaluedWeylNonTrivial}
\forall \ h \in \Inertia_{\QQ_{p}}, \ \tau_{{\weyl}}(a,b)(h) = \left(\begin{pmatrix}
\omega_{2}(h)^{a + 1 - pa} & 0\\
0 & \omega_{2}(h)^{b + p(1 - b)}
\end{pmatrix}, \omega_1(h)\right) \ .
\end{equation}
\end{Deff}

We can now define genericity for ${}^C\hat{G}$-valued inertial types as follows.
\begin{Deff}
\label{DefGenericCvaluedType}
Let $n$ be a positive integer. 
A ${}^C\hat{G}$-valued inertial type $\tau : \Inertia_{\QQ_{p}} \to {}^C\hat{G}$ is called {\it $n$-generic} {if} there exists a triple $(w,a,b) \in W \times \Z^{2}$ such that $\tau \simeq \tau_{w}(a,b)$ and $n < a - b + 1 < p - n$.
\end{Deff}

\begin{rem}
The pair of integers $(a,b)$ introduced here {appear because they give} a way to parameterise the characters of a split maximal torus of $\GL_2$. The notion of $n$-genericity in \cite{KM} is given in terms of the $n$-depth of a such a character, as defined in \cite[Definition 3.2]{KM}, and a direct calculation shows that
this notion is equivalent to Definition \ref{DefGenericCvaluedType}. For more details on these constructions, the reader is invited to read \cite[Sections 2.2.3 and 4.1.4]{KM}.
\end{rem}

\subsubsection{Genericity for $C$-parameters and their inertial types}
{We now use Definition \ref{DefGenericCvaluedType} to define what it means for a $p$-modular $C$-parameter to be $n$-generic.} First we must set up some additional notation. Let $\rho : \Gamma_{\QQ_{p}} \to {}^CG$ be a $C$-parameter. Then by the same logic as in Section \ref{section-Cparams-def}, the restriction $\rho \vert_{\Inertia_{\bQ_p}}$ is of the form $\rho(h) = (\rho_0(h), h)$, where $\rho_0$ is a $^C\hat{G}$-valued inertial type. By abuse of notation we write $\rho_0$ as $\rho\vert_{\Inertia_{\bQ_p}}$. 
\begin{Deff}
\label{DefCparamGener}
Let $n$ be a nonnegative integer. A $C$-parameter $\rho : \Gamma_{\QQ_{p}} \to {}^CG$ is {\it $n$-generic} if there exists an element $w \in W$ and a pair of integers $(a,b)$ such that $n < a-b+1 < p-n$ and $\rho\vert_{\Inertia_{\bQ_p}} \simeq \tau_w(a, b)$.
\end{Deff}

Note that $n$-genericity is determined solely by the restriction of parameters to the inertia subgroup $\Inertia_{\QQ_{p}}$. The next lemma tells us that $n$-generic non-supercuspidal Langlands parameters (as described in Section \ref{section-Cparams-def}) must correspond to trivial Weyl parameter.

\begin{Lem}
\label{LemTrivialWeylEltPSType}
Let $(r,\lam) \in \Z \times \FFF_{p}^{\times}$ with $0 \leq r \leq p^{2}-2$ and $n \geq 0$. If the non-supercuspidal $C$-parameter $\psi_{r, \lambda} : \Gamma_{\Q_{p}} \to {}^CG$ is $n$-generic, then there exists a pair $(a,b) \in \Z^{2}$ such that $n < a-b+1 < p-n$ and
\[\displaystyle \psi_{r,\lambda} \vert_{\Inertia_{\Q_{p}}} \simeq \tau_{1}(a,b) \ . \]
\end{Lem}
\begin{proof}
By construction, we know that for any $h \in \Inertia_{\Q_{p}}$, we have
 \begin{equation}
 \label{FormRestrictionNonsupercuspInertial}
 \displaystyle \psi_{r, \lam}\vert_{\Inertia_{\bQ_p}}(h) = \left(\begin{pmatrix}
\omega_2(h)^{p + 1 + r} & 0\\
0 & \omega_2(h)^{-pr}
\end{pmatrix},
\omega_1(h)\right) \ . 
\end{equation}
Assume that $\psi_{r,\lambda}$ is $n$-generic, which means that there exists a pair $(a,b)$ of integers and an element $w \in W$ such that $\psi_{r,\lambda}\vert_{\Inertia_{\Q_{p}}} \simeq \tau_{w}(a,b)$ with $n < a - b + 1 < p-n$. As $W = {\langle \weyl \rangle} \simeq \Z/2\Z$, we only have to prove that $w {\neq \weyl}$. We do this by contradiction, and assume that {$w = \weyl$}. Comparing \eqref{EquaTypeIntertielCvaluedWeylNonTrivial} to \eqref{FormRestrictionNonsupercuspInertial} shows that we must then satisfy one of the two following cases.

{\bf Case 1:} We have $\omega_{2}^{p + 1 + r} = \omega_{2}^{a + 1 - pa}$ and $\omega_{2}^{-pr} = \omega_{2}^{b + p(1-b)}$, which implies that we must have
\begin{eqnarray*} 
\nonumber p + r  &\equiv& a(1 - p) \text{ mod } p^{2} - 1 \text{ and }\\ 
\nonumber -pr &\equiv& b(1-p) + p \text{ mod } p^{2} - 1.
\end{eqnarray*}
The first congruence gives that $r \equiv a(1-p) -p \text{ mod } p^{2} - 1$. {Plugging into} the second congruence then shows that we must have $p^{2}(a + 1) - pa \equiv b + p(1 - b) \text{ mod } p^{2} - 1$, which can be rewritten as 
$a - pa + 1 \equiv b + p - pb \text{ mod } p^{2} - 1$,  i.e. as  
\begin{equation*}
a - b + 1 \equiv p(a - b + 1) \text{ mod } p^{2} - 1 \ .
\end{equation*}
The latter congruences implies that $p^{2} - 1$ divides $(p-1)(a-b+1)$, hence that $p+1$ divides $a-b+1$. Write $a-b+1 = m(p+1)$ with {$m \in \bZ$. The $n$-genericity} condition then implies that $n < m(p+1) < p-n$. Having $0 \leq n < m(p+1)$ shows that $m$ is positive, hence $m(p+1)$ is at least $p+1 > p \geq p-n$, which contradicts the fact that $m(p+1) < p-n$. Case 1 hence cannot occur.

{\bf Case 2:} We have $\omega_{2}^{-pr} = \omega_{2}^{a + 1 - pa}$ and $\omega_{2}^{p + 1 + r} = \omega_{2}^{b + p(1-b)}$. Then a {calculation similar to those} in Case 1 show that {$a + pb +1$ is} divisible by $p+1$, hence $a - b + 1 = a + pb + 1 - (p+1)b$ must also be divisible by $p+1$, but we proved above that this cannot occur. 

As neither of these cases can occur, $w$ cannot be equal to {$\weyl$} and the lemma is proven.
\end{proof}

\begin{rem} Note that the non-trivial Weyl element $s_{0}$ does not appear as parameter when considering $n$-generic non-supercuspidal $C$-parameters. Due to \cite[Corollary 6.16]{Koziol}, we suspect that types indexed by $s_{0}$ may be related to $n$-generic supercuspidal parameters, but this is work in progress.
\end{rem}
Now assume that $B$ is a finite local $E$-algebra and $\rho : \Gamma_{\QQ_{p}} \to {}^C\GG(B)$ is an $B$-valued $C$-parameter. Following \cite[Section 3.2]{BellovinGee}, we can define its inertial type using the associated {Weil--Deligne} representation $\WD(\rho)$.
\begin{Deff}
\label{DefInertialTypeCparam}
Let $\tau : \Inertia_{\QQ_{p^{2}}} \to \rm{GL}_{2}(\OO)$ be an inertial type and $\rho: \Gamma_{\QQ_{p}} \to {}^C\GG(B)$ be a $C$-parameter. We say that {\it $\rho$ has inertial type $\tau$} if $\WD(\rho)|_{\Inertia_{\QQ_{p^{2}}}} \simeq \tau \oplus {\bf 1}_{\Inertia_{\QQ_{p^{2}}}}$, where ${\bf 1}_{\Inertia_{\QQ_{p^2}}}$ denotes the trivial character of $\Inertia_{\QQ_{p^{2}}}$.
\end{Deff}

\
Under some mild conditions, we can find the inertial type of $\rho$ by looking at the so-called base change of $\rho$, which is a linear representation of $\Gamma_{\QQ_{p^{2}}}$ defined as follows.
\begin{Deff} 
\label{DefBCrho}
Let $R$ be a topological $\ZZ_{p}$-algebra $R$, and let $\rho : \Gamma_{\QQ_{p}} \to \ ^C\GG(R)$ be a $C$-parameter. Write the restriction of $\rho$ to $\Gamma_{\QQ_{p^{2}}}$ as
\[ \rho \vert_{\Gamma_{\QQ_{p^{2}}}}= \rho_{2} \oplus \rho_{1}  \text{ with }\rho_{2}: \Gamma_{\QQ_{p^{2}}} \to \GL_{2}(R) \text{ and } \rho_{1}:\Gamma_{\QQ_{p^{2}}} \to R^{\times} \ . \]
We then say that $\rho_{2}$ is the {\it base change of $\rho$}, and we write $\BC(\rho) := \rho_{2}$.
We also define the {\it multiplier of $\rho$} as the composite character 
\[
\Gamma_{\QQ_{p}} \stackrel{\rho}{\longrightarrow} {}^C\GG(R) \stackrel{\widehat{\iota}}{\longrightarrow} R^\times \ .
\]
where $\hat{\iota}$ is as defined in \cite[Section 2.3.3]{KM}.

Note that $\rho_{1}$ is basically the restriction to $\Gamma_{\QQ_{p^{2}}}$ of the multiplier of $\rho$.
In particular, we say that $\rho$ has {\it cyclotomic multiplier} if $\rho_{1}$ is the (reduction mod $p$ of the) cyclotomic character $\Gamma_{\QQ_{p}} \to \ZZ_{p}^{\times}$.
\end{Deff}

The notions of base change and of inertial type interact in the following proposition, which follows from \cite[Section 5.3.3]{KM}.

\begin{Prop}
\label{PropIntertialTypeChar0}
Let $\rho : \Gamma_{\QQ_{p}} \to {}^C\GG(B)$ be a $C$-parameter with cyclotomic multiplier and let $\tau : \Inertia_{\QQ_{p^{2}}} \to \GL_{2}(\OO)$ be a principal series inertial type.
Then $\rho$ has inertial type $\tau$ if, and only if, we have $\WD(\BC(\rho))\vert_{\Inertia_{\QQ_{p^{2}}}} \simeq \tau$.
\end{Prop}

We thus have several notions of genericity appearing in the context of $C$-parameters, so we wonder to which extent they are compatible. In particular, we are interested in how compatible $n$-genericity is with reduction modulo $p$, or with deformation theory as defined in Section \ref{SubsectionDeformationGeneralites} below. In particular, we hence aim to solve the following question. 
\begin{What}
Given a $B$-valued $C$-parameter $\rho : \Gamma_{\QQ_{p}} \to {}^C\GG(B)$ with principal series inertial type $\tau$, let $\bar{\rho} : \Gamma_{\QQ_{p}} \to {}^CG$ be its reduction modulo $p$. Is $n$-genericity for $\rho$ (i.e. for $\tau$, in the sense of Definition~\ref{DefGenericityInertialTypes}) equivalent to $n$-genericity for $\bar{\rho}$ (in the sense of Definition \ref{DefCparamGener})? 
\end{What}

\subsection{Deforming Galois parameters}\label{SubsectionDeformationGeneralites}
As we did for $p$-modular automorphic representations in Section \ref{section-nonsc-deformations}, we now want to deform $p$-modular $C$-parameters. We will do so in the framework of deformation theory as first introduced by Mazur in this setting \cite{Mazur88}, then later developed by many authors. We start by recalling some basic facts and definitions. We then specialise to specific deformations that are not only more convenient to handle, since they can be quite well-understood via some (advanced and technical) semi-linear algebra that will be introduced in the next subsections, but also do carry some interesting geometric and arithmetic information related to Langlands correspondences. In particular, we will introduce the notion of a Hodge--Tate type, which relates to the classical notion of Hodge--Tate weights, and make more assumptions on the inertial types we will use. 

Though it may not be obvious at first sight, these notions and assumptions we make on the Galois side have a natural counterpart on the automorphic side, as can be seen in the model case of $\GL_{2}$. In this setting, a first automorphic interpretation of these data is given by the Breuil--M\'ezard conjecture \cite[Conjecture 1.1]{BM02}. This conjecture has the following motivation: starting from a continuous $p$-modular two-dimensional representation $\bar{\rho}$ of $\AbGal_{\QQ_{p}}$, one wants to understand the deformations of $\bar{\rho}$ with prescribed inertial type $\tau$ and Hodge--Tate weights $(0, k-1)$. The Breuil--M\'ezard conjecture predicts that the ring parametrising these deformations can be (at least partially) understood through the study of the semisimplification of the representation $\sigma(\tau) \otimes \mathrm{Sym}^{k-2}((\overline{\QQ}_{p})^{2})$ of $\GL_{2}(\ZZ_{p})$, where $\sigma(\tau)$ is uniquely determined by $\tau$ via the local Langlands correspondence for $\GL_{2}(\QQ_{p})$ \cite[Section A.1.5]{BM02}. 

Another automorphic interpretation, related to the previous one, explicitly appears in the origin of the $p$-adic Langlands program. Indeed, one of the core statement of this program is \cite[Th\'eor\`eme~5.1]{BreuilBourbaki12}, where one sees that the connection between the smooth and algebraic representations attached to a given (potentially semi-stable) Galois representation $\rho_{p}$ can only be made via an automorphic interpretation of its Hodge--Tate weights, again with the representation $\mathrm{Sym}^{k-2}((\overline{\QQ}_{p})^{2})$ that showed up earlier (and is usually called a {\it Serre weight} in this context).

\subsubsection{Universal framed deformations} 
\label{subsubsec:UnivDefTheory}
Recall that $\mathrm{Noe}(\OO)$ is the category of complete Noetherian local $\Oe$-algebras with residue field $k$. For any $A$ in $\mathrm{Noe}(\OO)$, we write $\phi_{A}: A \to k$ for the reduction map.

Let $\Profin$ be a profinite group, and let $\bar{\rho} : \Gamma \to {}^C\GG(k)$ be a continuous group homomorphism. 
\begin{Deff}
The {\it functor of framed deformations of $\bar{\rho}$} is the functor $D_{\bar{\rho}}^{\square}$ 
that associates to each $A$ in $\mathrm{Noe}(\OO)$ the set of continuous representations $\rho_{A}: \Profin \to {}^C\GG(A)$ such that the composition
\[
\Profin \stackrel{\rho_A}{\longrightarrow}  {}^C\GG(A)  \stackrel{\phi_A}{\longrightarrow} {}^C\GG(k) \ 
\]
is equal to $\bar{\rho}$. We say that $D_{\bar{\rho}}^{\square}(A)$ is the set of {\textit{framed deformations}} of $\bar{\rho}$ to $A$.
\end{Deff}
Unless we add extra assumptions on $\Gamma$ and/or on $\bar{\rho}$, there is a priori no reason ensuring that the functor $D_{\bar{\rho}}^{\square}$ is representable. In our work, we only consider the case where $\Gamma$ is the absolute Galois group $\Gamma_{F}$ with $F$ being either $\QQ_{p}$ or $\QQ_{p^{2}}$. In this setting, we have the following nice result, which directly follows from \cite[Theorem 1.2.2]{Balaji}.

\begin{Thm}
Let $F$ be a finite extension of $\QQ_{p}$ and $\bar{\rho}: \Gamma_{F} \to {}^C\GG(k)$ be a continuous representation. Then the framed deformation functor $D_{\bar{\rho}}^{\square}$ is representable by an object $R_{\bar{\rho}}^{\square}$ of $\mathrm{Noe}(\OO)$.
\end{Thm}

Langlands parameters that naturally appear in the Langlands program
(for instance, arising in the cohomology of Shimura varieties or coming from automorphic forms)
satisfy additional properties, typically being potentially semi-stable or crystalline.
We will then focus on deformations of a fixed $C$-parameter modulo $p$ with this kind of additional conditions.
They are described by some quotients of the universal ring $R_{\bar{\rho}}^{\square}$ that will be introduced in Section \ref{sssec:def_Galois_potcrystype}.
Note that the geometry of these quotients is the subject of the Breuil--M\'{e}zard conjecture mentioned above;
they also play an important role in global modularity results and are related with Serre weights conjectures.  
 
\subsubsection{Intermission: Frobenius-twist self-dual inertial types}
Recall that we have fixed a geometric Frobenius $\Frob$ (see Section \ref{section-intro}). Note that our definition implies that $\varphi^{-1}$ is an  arithmetic Frobenius. 
\begin{Deff}
\label{DefFTSDtypes}
Given a principal series inertial type $\tau$ of the form $\eta_1 \oplus \eta_2$, we call $\tau^{\vee} := \eta_1^{-1} \oplus \eta_2^{-1}$ the {\it dual type} of $\tau$ and $\sigma^{*}\tau := \eta_1^{p^{-1}} \oplus \eta_2^{p^{-1}}$ the {\it Frobenius-twist of $\tau$}. When $\sigma^{*}\tau \simeq \tau^{\vee}$, we say that $\tau$ is {\it Frobenius-twist self-dual}. 
\end{Deff}
Principal series inertial types that are Frobenius-twist self-dual can be described explicitly. Indeed, if $\tau = \widetilde{\omega}_{2}^{a} \oplus \widetilde{\omega}_{2}^{b}$, then $\tau^{\vee} = \widetilde{\omega}_{2}^{-a} \oplus \widetilde{\omega}_{2}^{-b}$ and $\sigma^{*}\tau = \widetilde{\omega}_{2}^{p^{-1}a} \oplus \widetilde{\omega}_{2}^{p^{-1}b}$. A direct calculation shows that such a type $\tau$ is Frobenius-twist self-dual if, and only if, there exists integers $a,b$ such that $\tau$ is either of the form $\widetilde{\omega}_{2}^{(1-p)a} \oplus \widetilde{\omega}_{2}^{(1-p)b}$ or of the form $\widetilde{\omega}_{2}^{a} \oplus \widetilde{\omega}_{2}^{-pa}$. Note that, according to \cite[Section 4.4]{KM}, this is also equivalent to asking that $\tau$ is given by the restriction to $\Inertia_{\QQ_{p^{2}}}$ of the base change of an $\OO$-valued $C$-parameter $\rho : \Gamma_{\QQ_{p}} \to {}^C\GG(\OO)$ with trivial multiplier.

\subsubsection{Potentially crystalline deformations with prescribed Hodge type and inertial type}
\label{sssec:def_Galois_potcrystype}
In this subsection, we let $F$ denote either $\QQ_{p}$ or $\QQ_{p^{2}}$ and we let $B$ be any finite local $E$-algebra. We first transfer to $C$-parameters the classical conditions imposed to genuine Galois representations arising in Langlands program, namely being potentially crystalline %/semi-stable
 and of Hodge type $(\underline{1},\underline{0},\underline{1})$. For a review of these notions in the classical setting, the reader is invited to read for instance \cite{ColmezECM}.
\begin{Deff}
\label{DefPotenCristParal}
A $C$-parameter $\Gamma_{F} \to {}^C\GG(B)$ is {\it potentially crystalline} %(resp. {\it potentially semi-stable)
 if, and only if, its composition with any faithful algebraic representation $^C\GG \xhookrightarrow{} \GL_{n}$ is potentially crystalline in the classical sense of \cite[5.1.4]{Font88}.
\end{Deff}

\begin{Deff}
Assume that $\rho: \Gamma_{F} \to {}^C\GG(B)$ is a $C$-parameter with cyclotomic multiplier. We say that $\rho$ has {\it $p$-adic Hodge type} $(\underline{1},\underline{0},\underline{1})$ when its base change $\BC(\rho)$ has $p$-adic Hodge type $(\underline{1}, \underline{0})$, which means that $\BC(\rho)$ has Hodge-Tate weights $\{-1,0\}$.
\end{Deff}

Given a principal series inertial type $\tau : \Inertia_{\QQ_{p^{2}}} \to \GL_{2}(\OO)$ that is Frobenius-twist self-dual (in the sense of Definition \ref{DefFTSDtypes}), we are interested in lifts of $\rhobar$ that are potentially crystalline with $p$-adic Hodge type $(\underline{1},\underline{0},\underline{1})$, inertial type $\tau$ and cyclotomic multiplier.
They are parameterised by a quotient of the universal framed deformation ring $R_{\bar{\rho}}^{\square}$.
The existence and properties of this quotient have been established in our case in \cite[Sections 3.2 and 3.3]{BellovinGee} (see also the proof of \cite[Proposition 3.0.12]{Balaji} and \cite[Section 5.3.2]{KM}). In the following theorem, $\rho^\mathrm{u}$ is the universal $C$-parameter from $\Gamma_{\Q_p}$ to ${}^C\GG(R_{\bar{\rho}}^{\square})$ lifting $\rhobar$.

\begin{Thm}
\label{Def:R_rob_tau} 
Let $\tau : \Inertia_{\QQ_{p^{2}}} \to \GL_{2}(\OO)$ be a principal series inertial type that is Frobenius-twist self-dual.
 Then there exists a unique quotient $R_{\bar{\rho}}^{\tau}$  of $R_{\bar{\rho}}^{\square}$ satisfying the following property:
  for any finite local $E$-algebra $B$, and any morphism $x$ from $R_{\bar{\rho}}^{\square}$ to $B$,  
  $x$ factors through $R_{\bar{\rho}}^{\tau}$ if and only if the $C$-parameter $x \circ \rho^{u}$ from $\Gamma_{\Q_p}$ to ${}^C\GG(B)$ is  potentially crystalline with $p$-adic Hodge type $(\underline{1},\underline{0},\underline{1})$, inertial type $\tau$ and cyclotomic multiplier.
\end{Thm}

\subsection{From $C$-parameters to Kisin modules}
Using $p$-adic Hodge theory, the potentially crystalline lifts we described in the previous section can be described by objects of semi-linear algebra, which are easier to handle.
First, we will explain in Section \ref{SectionParametersTwdGalois} how we have, on the one hand, a natural bridge between $C$-parameters and genuine two-dimensional Galois representations enriched with extra structures. Doing so, we will give an explicit statement for the bijection announced in Remark \ref{RemConnectionCparChar0GalRep}. Then, we will explain how to relate (most of) these Galois representations to a category of modules (called Kisin modules) over an appropriate power series ring, which motivates why we get interested in deformations of such modules in the next subsection of this paper.

\subsubsection{From $C$-parameters to genuine $p$-adic Galois representations}
\label{SectionParametersTwdGalois} 
Let $R$ be a topological $\ZZ_{p}$-algebra. The notion of base change introduced in Definition \ref{DefBCrho} is a first step to connect $R$-valued $C$-parameters to two-dimensional $R$-linear Galois representations, but it is clearly not enough as non-equivalent $C$-parameters could have isomorphic base changes. Going further hence requires extra structures given by the notion of polarisation, which we define now.
\begin{Deff}
\label{Def:pol_Galois}
Let $\tilde{\rho} : \Gamma_{\QQ_{p^{2}}} \to \GL_{2}(R)$ be a continuous group homomorphism and $\theta : \Gamma_{\QQ_{p}} \to R^{\times}$ be a continuous character. A {\it polarisation of $\tilde{\rho}$ compatible with $\theta$} is an isomorphism $\alpha : \tilde{\rho}^{\varphi^{-1}} \stackrel{\sim}{\to} \tilde{\rho}^{\vee} \otimes \theta$ that is such that the composite map
\[
\displaystyle \tilde{\rho} \xrightarrow{\tilde{\rho}(\varphi^{-2})} \tilde{\rho}^{\varphi^{-2}} \xrightarrow{\alpha^{\varphi^{-1}}} \left(\tilde{\rho}^{\vee} \otimes \theta\right)^{\varphi^{-1}} \xrightarrow{\mathrm{can}} \left( \left(\tilde{\rho} \otimes \theta^{-1}\right)^{\varphi^{-1}} \right)^{\vee} \xrightarrow{\left( \left(\alpha \otimes \theta^{-1}\right)^{\vee} \right)^{-1}} \tilde{\rho}
 \]
is the multiplication by $-\theta(\varphi^{-1})$ map.
\end{Deff}
Considering triples $(\tilde{\rho}, \theta, \alpha)$ as in Definition \ref{Def:pol_Galois} is now enough to distinguish between non-equivalent $C$-parameters, as stated by the next theorem.
\begin{Thm}
\label{EquivCparamPolarisedTriples}  
Let $\rho : \Gamma_{\QQ_{p}} \to {}^C\GG(R)$ be a $C$-parameter with multiplier $\rho_{1}$. 
\begin{itemize}
\item[\tt $(i)$] Let $A$ be the $\GL_{2}(R)$-component of $\rho(\varphi^{-1})$ and
 let $\alpha$ be the $R$-endomorphism of $R^{2}$ defined by $\alpha(v) := \begin{psmallmatrix}
 0 & -1 \\
 1 & 0 \\
\end{psmallmatrix}
A^{-1}v $. Then $\alpha$ is a polarisation of $BC(\rho)$ compatible with $\rho_{1}$.
\item[\tt $(ii)$] The previous construction induces a bijection $\rho \mapsto (BC(\rho), \rho_{1}, \alpha)$ from $R$-valued $C$-parameters to triples $(\tilde{\rho}, \theta, \alpha)$ where $\tilde{\rho}$ denotes a continuous morphism from $\Gamma_{\QQ_{p^{2}}}$ to $\GL_2(R)$,
 $\theta$ a continuous $R$-character of $\Gamma_{\QQ_{p}}$ and $\alpha$ a polarisation of $\tilde{\rho}$ compatible with $\theta$.
\end{itemize}
\end{Thm}
\begin{proof}
This is \cite[Lemma 2.1.1]{CHT08} rephrased in the language of \cite[Section 5.3.2]{KM}.
\end{proof}

From now on, we fix $\theta$ to be the cyclotomic character. This implies in particular that the bijection of Theorem \ref{EquivCparamPolarisedTriples} actually stands between $C$-parameters $\rho : \Gamma_{\QQ_{p}} \to {}^C\GG(R)$ with cyclotomic multiplier and pairs $(\tilde{\rho}, \alpha)$ where $\tilde{\rho} : \Gamma_{\QQ_{p^{2}}} \to \GL_{2}(R)$ is a continuous representation of $\Gamma_{\QQ_{p^{2}}}$ and $\alpha$ is a polarisation of $\tilde{\rho}$ compatible with the cyclotomic character (note that in that case $-\varepsilon(\varphi^{-1}) = -1$).

\subsubsection{Kisin modules with prescribed descent data and height}
\label{SectionDefinitionKisinModules} 
Thanks to Theorem \ref{EquivCparamPolarisedTriples}, we are reduced to understand genuine Galois representations endowed with some extra structure.
To do this, we introduce a category of modules over a power series ring, called {\it Kisin modules}, which allow to translate arithmetic and geometric properties into semi-linear algebra. These modules come with decorations that reflect the extra data that appeared on our Galois representations, and with a natural functor from Kisin modules to Galois representations that preserves these decorations but is not, in general, an equivalence of categories (see Section \ref{SubsectionFromKisinToGalois}). Though the construction of decorated Kisin modules is a bit technical, it is really useful since these modules and their deformations are much more handy to compute than the Galois representations they parametrise.

From now on, we let $R$ denote a complete local Noetherian $\mathcal{O}$-algebra with residue field $k$ and we set $\mathfrak{S}_{R} := \left( \Z_{p^{2}} \otimes_{\Z_{p}} R\right) [[ u ]]$.
The ring $\mathfrak{S}_{R}$ is equipped with a Frobenius endomorphism $\overline{\varphi}$ that is trivial on $R$, sends $u$ to $u^p$, and is the \emph{arithmetic} Frobenius on $\Z_{p^2} $
(i.e. satisfies $\phibar = \varphi^{-1}$ on $\Z_{p^2}$).

\begin{Deff}  
A \emph{Kisin module over $R$ with height in $[0,1]$} is a pair $(\mathfrak{M}, \phi_\mathfrak{M})$, where $\MM$ is a finitely generated projective $\SSSS_{R}$-module, and $\phi_{\MM} : \phibar^{*}\mathfrak{M} := \mathfrak{S}_{R} \otimes_{\mathfrak{S}_{R}, \phibar} \mathfrak{M} \to \mathfrak{M}$ is an $\SSSS_{R}$-linear map, that satisfies
\begin{equation}
\label{EqCaracModKisin}
(u^{p^2-1} + p)\mathfrak{M} \subseteq  \phi_\mathfrak{M} \left( \phibar^*\mathfrak{M}  \right) \subseteq \mathfrak{M} \ .
\end{equation}
\end{Deff}

We let $Y^{[0,1]}(R)$ be the category of Kisin modules over $R$ with height in [0,1].
Given an object $(\mathfrak{M}, \phi_\mathfrak{M})$ in $Y^{[0,1]}(R)$ and an integer $i$ in $\{0,1\}$, we define $\frakM^{(i)}$ as the $R[[u]]$-submodule of $\frakM$ on which $\Z_{p^2}$ acts through the embedding $\sigma_0 \circ \varphi^i$, where $\sigma_{0}$ still denotes the embedding of $\Q_{p^2}$ into $E$ we fixed initially:
\begin{equation}
\label{Eqn:factorKM}
 \frakM^{(i)} = \left\{ m \in \frakM \,|\,\, \forall x \in \Z_{p^2}, (x \otimes 1_R)m = (1_{\Z_{p^2}} \otimes \left(\sigma_0 \circ \varphi^i)(x) \right) m \right\} \ .
 \end{equation}
Note that, as an $R[[u]]$-module, $\frakM$ decomposes as a direct sum $\frakM^{(0)} \oplus \frakM^{(1)}$. 

Recall that $\pi$ is a fixed $(p^{2} - 1)$-th root of $-p$ in $E$ and that we set $L = \Q_{p^2}(\pi)$.
 Given $g$ in $\Gal(L / \Q_{p^2})$, we define $\widehat{g}$ as the $\Z_{p^{2}} \otimes_{\Z_{p}} R$-linear automorphism of $\mathfrak{S}_R$ that sends $u$ to $\left(\widetilde{\omega}_{\pi}(g) \otimes 1_R \right)u$. Note that we have, for all $g$, $h$ in  $\Gal(L / \Q_{p^2})$: $\widehat{gh} =  \widehat{g} \circ \widehat{h} $ and $ \overline{\varphi} \circ \widehat{g} = \widehat{g} \circ  \overline{\varphi} $. Further note that for any $i \in \{0, 1\}$, both $\MM^{(i)}$ and $u\MM^{(i)}$ are stable under the action of $\Gal(L/\QQ_{p^{2}})$ on $\MM$, which ensures that the following definition makes sense.

\begin{Deff}%[KM def 5.2 + 5.3]
Let $(\MM, \phi_{\MM})$ be an object of $Y^{[0,1]}(R)$ such that $\MM$ is of rank $2$ as an $\SSSS_{R}$-module, and let $\tau$ be a principal series inertial type. A \emph{descent data of type $\tau$ on $\frakM$} is a collection of $\Z_{p^{2}} \otimes_{\Z_{p}} R$-linear automorphisms $(\widecheck{g})_{g \in \Gal(L / \Q_{p^2})}$ of $\MM$ that satisfies the following conditions:
\begin{enumerate}[(i)]
\item for any $g$ in $\Gal(L / \Q_{p^2}) $, $\widecheck{g}$ is  $\widehat{g}$-semilinear;
\item for all $g$ and $h$ in $\Gal(L / \Q_{p^2}) $, one has $\widecheck{gh} =  \widecheck{g}  \circ \widecheck{h} $;
\item for any $g$ in $\Gal(L / \Q_{p^2}) $, $\widecheck{g}  \circ \phi_\frakM = \phi_\frakM \circ \phibar^*\widecheck{g}$;
\item for any $i$ in $\{ 0,1\}$, the $R[\Gal(L/\QQ_{p^{2}})]$-module $\frakM^{(i)} / u\frakM^{(i)} $ is isomorphic to $\tau \otimes_{\mathcal{O}} R$.
\end{enumerate}
\end{Deff} 
%KM def 5.4
We let $Y^{[0,1],\tau}(R)$ be the full subcategory of $Y^{[0,1]}(R)$ whose objects are rank $2$ modules having descent data of type $\tau$. Finally, we define a full subcategory of $Y^{[0,1], \tau}(R)$ that is the category of Kisin modules we are really interested in. 
\begin{Deff} 
We let $Y^{\tau}(R)$ be the full subcategory of $Y^{[0,1], \tau}(R)$ whose objects are Kisin modules whose determinant satisfy the following chain of inclusions:
\[ \left(u^{p^2-1} + p \right) \det \frakM \subseteq   \phi_\mathfrak{M} \left( \phibar^*\left(\det \mathfrak{M}   \right)   \right)  \subseteq \det \frakM \ . \]
\end{Deff}
Note that this is the category denoted by $Y^{\mu, \tau}(R)$ in \cite[Definition 5.4]{KM}.

Below we give explicit examples of Kisin modules in $Y^{\tau}(R)$ (see Sections~\ref{sssec:shape} and~\ref{sssec:deformKM}). Note that full classification results for the objects of $Y^{\tau}(R)$ are actually available: see Proposition~\ref{prop:shape}, Definition~\ref{Def:gauge_basis} and Remark~\ref{Rem:gauge_basis}, as well as \cite[Proposition~3.1.9]{CDM}.
They give in particular a very concrete description of the action of the Frobenius morphism $\phi_\mathfrak{M}$ in some bases adapted to the descent data.

\subsubsection{Frobenius-twist self-dual Kisin modules and associated Galois representations} 
\label{SubsectionFromKisinToGalois}
As for inertial types, we have a notion of Frobenius-twist self-duality for Kisin modules. Let $\sigma$ denote the automorphism of $\SSSS_{R}$ that is given by the identity map on $R$, fixes $u$, and is given by the arithmetic Frobenius $\varphi^{-1}$ on $\ZZ_{p^{2}}$. Given a Kisin module $(\MM, \phi_{\MM})$ in $Y^{\tau}(R)$, its pullback $\sigma^{*}\MM$ by $\sigma$ defines an object of $Y^{\sigma^{*}\tau}(R)$, where $\sigma^{*}\tau$ is the Frobenius-twist of $\tau$ introduced in Definition \ref{DefFTSDtypes}. We obtain this way a functor from $Y^{\tau}(R)$ to $Y^{\sigma^{*}\tau}(R)$, which can be iterated via successive pullbacks by $\sigma$.
On the other hand, if $R$ is local Artinian, we define the {\it Cartier dual of $\MM$} as $\MM^{\vee} := \Hom_{\SSSS_{R}}(\MM, \SSSS_{R})$. The next proposition ensures that this definition gives what one can expect from Cartier duality in this context, and shows in particular that it is compatible with duality for inertial types.
\begin{Prop}
Assume that $R$ is Artinian. Then the map sending $\MM$ to $\MM^{\vee}$ induces an involutive functor from $Y^{\tau}(R)$ into $Y^{\tau^{\vee}}(R)$.
\end{Prop}
\begin{proof}
This comes from the first statement in \cite[Proposition 5.13]{KM}.
\end{proof}

In particular, if $\tau$ is Frobenius-twist self-dual, then $\sigma^{*}\MM$ and $\MM^{\vee}$ both belong to $Y^{\tau^{\vee}}(R) = Y^{\sigma^{*}\tau}(R)$.
This enables the following definition of polarisation for Kisin modules, which actually mimics the corresponding notion for continuous two-dimensional $R$-linear representations of $\Gamma_{\QQ_{p^{2}}}$ (Definition \ref{Def:pol_Galois}).
\begin{Deff} % KM Def 5.15 
Assume that $R$ is Artinian and that $\tau$ is Frobenius-twist self-dual, and let $(\MM, \phi_{M})$ be an object of $Y^{\tau}(R)$. A {\it polarisation on the Kisin module $\MM$} is an isomorphism $\iota : \sigma^{*}\MM \stackrel{\sim}{\longrightarrow} \MM^{\vee}$ in $Y^{\tau^{\vee}}(R)$ such that the composite map 
\[ \displaystyle \frakM \stackrel{\mathrm{can}}{\longrightarrow} \sigma^{*} \left(\sigma^{*}\frakM \right) \stackrel{\sigma^{*}\iota}{\longrightarrow} \sigma^{*}\left(\frakM^{\vee}\right) \stackrel{\mathrm{can}}{\longrightarrow}  \left(\sigma^{*}\frakM\right)^{\vee} \stackrel{(\iota^{\vee})^{-1}}{\longrightarrow} \frakM \ . \]
is equal to $-\id_{\MM}$.
\end{Deff}
We let $Y^{\tau}_{\mathrm{pol}}(R)$ be the category of {\it Frobenius-twist self-dual Kisin modules (of type $\tau$)}: its objects are pairs $(\MM, \iota)$ with $\MM$ an object of $Y^{\tau}(R)$ and $\iota$ a polarisation on $\MM$, and morphisms in $Y^{\tau}_{\mathrm{pol}}(R)$ are given by morphisms in $Y^{\tau}(R)$ that commute with the given polarisations.

Frobenius-twist self-dual Kisin modules are closely related to polarised Galois representations, hence to $C$-parameters, introduced in Definition \ref{Def:pol_Galois}. Indeed, let us fix a sequence of compatible $p^{n}$-th roots of $-p$ in $\QQQ_{p}$, which means a sequence $(p_{n})_{n \geq 0}$ of elements of $\QQQ_{p}$ such that $p_{0} = -p$ and $p_{n+1}^{p} = p_{n}$ for any integer $n \geq 0$. For any $k \geq 1$, set 
\[\displaystyle \QQ_{p^{k}, \infty} := \bigcup_{n \geq 0} \QQ_{p^{k}}(p_{n}) \ . \]
Following \cite[Section 2.3]{S&S}, one can naturally define a contravariant functor $\Functor$ from $Y^{\tau}(R)$ to the category of $R$-linear representations of $\Gamma_{\Q_{p^{2}, \infty}}$. The following statement essentially claims that if $\tau$ is a reasonable type, any tamely ramified $p$-modular $C$-parameter with cyclotomic multiplier comes from at most one Kisin module in $Y^{\tau}(k)$.
\begin{Lem}
\label{Lem:varK_trivial}
Let $\rhobar: \Gamma_{\Qp} \to {}^CG$ be a tamely ramified $C$-parameter with cyclotomic multiplier and let $\tau : I_{\Q_{p^{2}}} \to \GL_{2}(\OO)$ be a $2$-generic principal series inertial type that is Frobenius-twist self dual.
Then there exists at most one Kisin module $\overline{\MM} \in Y^{\tau}(k)$ such that $\Functor\left(\overline{\MM}\right) \simeq \BC(\rhobar)\vert_{\Gamma_{\QQ_{p^{2}, \infty}}}$. Moreover, if such an $\overline{\MM}$ exists, then there exists a unique polarisation $\bar{\iota}$ on $\MM$ that is compatible (under the previous isomorphism) with the polarisation defined by $\rhobar$ through the bijection of Theorem \ref{EquivCparamPolarisedTriples}.
\end{Lem}
\begin{proof}
This is \cite[Lemma 5.20]{KM}, which is an analogue of \cite[Theorem 3.2]{S&S} for the unitary group $G$.
\end{proof}
Note that the existence of $\overline{\MM}$ as above is a necessary condition for the ring $R_{\rhobar}^{\tau}$ to be non-zero. Also note that the genericity assumption is what ensures the uniqueness of $\overline{\MM}$ when it exists. Hence, from now on, we will always assume that $\tau$ is a $2$-generic principal series Frobenius-twist self-dual inertial type.
Lemma \ref{Lem:varK_trivial} suggests that deforming Kisin modules may be 
a good way to approach deformations of $C$-parameters and the quite explicit nature of these modules suggests that deforming them may give rise to quite explicit rings. The next (and last) subsection will make these expectations a bit more precise under some genericity assumption on $\rhobar$.

\subsection{Some explicit deformation rings for $C$-parameters} 
The goal of this last part is to give some explicit calculations of deformation rings for deformations of $C$-parameters with prescribed inertial type.  We fix a principal series inertial type $\tau$ that is $2$-generic (in the sense of Definition \ref{DefGenericityInertialTypes}).
We write $\tau = \eta_{1} \oplus \eta_{2}$ and \textit{we fix the ordering of these two characters}. For convenience of writing, we set $v := u^{p^{2}-1}$.
\subsubsection{Shape of a Kisin module over $k$} \label{sssec:shape}
In order to deform objects of $Y^{\tau}(k)$, it would be helpful to have a classification of those.
A way to distinguish between them is to introduce the notion of shape of such modules: this is the goal of this first subsection, and it requires some preliminary notations an definitions that are valid over any object $R$ of $\mathrm{Art}(\OO)$.

Given a Kisin module $\MM \in Y^{[0,1], \tau}(R)$, we know from Section \ref{SectionDefinitionKisinModules} that the underlying $R[[u]]$-module decomposes as $\MM^{(0)} \oplus \MM^{(1)}$. To take into account the action of $\Gal(L/\QQ_{p})$ on these components (via $\tau$), we introduce the following $R[[v]]$-submodules of $\MM$.
\begin{Deff} %KM Def 5.7
\label{Def:isotypic_comp-KisinModule}
For $i$ in $\{ 0, 1 \} $  and $j$ in $\{ 1 , 2 \}$, we write $\MM^{(i)}_{j}$ for the $\eta_{j}$-isotypical component of $\MM^{(i)}$, i.e. for the $R[[v]]$-submodule of elements of $\MM^{(i)}$ on which $\Gal(L/\QQ_{p})$ acts by $\eta_{j}$. Similarly, we write ${}^{\overline{\varphi}}\frakM^{(i)}_{j}$ for the $\eta_{j}$-isotypical component of $\overline{\varphi}^{*}\left(\frakM^{(i)}\right)$
\end{Deff}
Since we have $\overline{\varphi}^{*}\left(\frakM^{(i)}\right) = \left(\overline{\varphi}^{*}\MM\right)^{(i+1)}$ (considering $i+1 \mod 2$ in the exponent if necessary), we can restrict $\phi_{\MM}$ to ${}^{\overline{\varphi}}\frakM^{(i)}_{j}$ to get a map $\phi_{\MM, j}^{(i)} : {}^{\overline{\varphi}}\frakM^{(i)}_{j} \to \MM^{(i+1)}_{j}$.

\begin{Deff} %KM Def 5.7 suite
\label{Def:EigenBasisKisinModule}
An {\it eigenbasis} of $\MM$ is a pair $\beta = \left( \beta^{(0)} , \beta^{(1)} \right)$ such that, for any $i$ in $\{0, 1\}$, $\beta^{(i)}$ is an ordered basis $\left( f^{(i)}_{1}, f^{(i)}_{2} \right)$ of the $R[[u]]$-module $\MM^{(i)}$ that satisfies: $f^{(i)}_{1} \in \MM^{(i)}_{1}$ and $f^{(i)}_{2} \in \MM^{(i)}_{2}$.
\end{Deff}

Just like above Definition \ref{DefGenericityInertialTypes}, we write the inertial type $\tau$ as $\widetilde{\omega}_2^{-\mathbf{a}_1}\oplus\widetilde{\omega}_2^{-\mathbf{a}_2}$, with $\mathbf{a}_1$ and $\mathbf{a}_2$ between $0$ and $p^2 - 2$ (recall that we have fixed the ordering of the two characters in $\tau$).
For $j$ in $\{0,1\}$, we write $\mathbf{a}_j$ in $p$-basis  : 
$$
\mathbf{a}_j = a_{j,0} +  a_{j,1}p, \text{ with } (a_{j,0},a_{j,1}) \in \{0 ; p-1\}^2
$$
To keep track of the action of $\tau$ on each of the factors $\MM^{(i)}$ (corresponding to the two embeddings of $\ZZ_{p^2}$ into $E$), we also define, for $j$ in $\{0,1\}$ :
$$
\left\{
\begin{array}{c c l}
\mathbf{a}_j^{(0)} &  =  &\mathbf{a}_j  \\
 \mathbf{a}_j^{(1)} &  = & a_{j,1} +  a_{j,0}p.\\
\end{array}
\right.
 $$
Recall that $W$ is the Weyl group of $G$, which is canonically isomorphic to $\SSSS_{2}$.

\begin{Deff} % KM Def 5.6
\label{DefOrientationInertialType}
An {\it orientation of $\tau$} is a pair $(w_{0}, w_{1})$ of elements of $W \simeq \SSSS_{2}$ such that : 
\begin{equation}
\label{IneqOrientationType}
\forall \ i \in \{0, 1\},\ \mathbf{a}_{w_{i}(1)}^{(i)} \geq \mathbf{a}_{w_{i}(2)}^{(i)} \ . 
\end{equation}
\end{Deff}
\begin{rem}
\label{RemOrientationUnique}
As $\tau$ is assumed to be $2$-generic, its orientation is uniquely defined. The $2$-genericity assumption (see \eqref{GenericityConditionPSType}) implies 
indeed that the inequalities \eqref{IneqOrientationType} must be strict.
\end{rem}
Remark \ref{RemOrientationUnique} ensures that we can talk about \textit{the} orientation $(w_{0}, w_{1})$ of $\tau$. We then have the following result, which gives convenient bases for the isotypical components associated with $w_{0}(2)$ and $w_{1}(2)$, and is straightforward to check from Definition \ref{Def:EigenBasisKisinModule}.
%\S 5.1.8 of KM
\begin{Prop} Let $\beta = \left( \beta^{(0)} , \beta^{(1)} \right)$ be an eigenbasis of $\MM$. With the notation of Definition \ref{Def:EigenBasisKisinModule}, we set\footnote{With $i-1$ replaced by $1$ in the second formula if $i = 0$, as usual.}, for all $i \in \{0, 1\}$, 
\[
\displaystyle
 \beta^{(i)}_{w_{i}(2)} := \left(u^{{\bf a}_{w_{i}(1)}^{(i)} - {\bf a}_{w_{i}(2)}^{(i)}}f^{(i)}_{w_{i}(1)},   f^{(i)}_{w_{i}(2)}\right) 
 \text{ and } 
 {}^{\overline{\varphi}}\beta^{(i-1)}_{w_{i}(2)} := \left(u^{{\bf a}_{w_{i}(1)}^{(i)} - {\bf a}_{w_{i}(2)}^{(i)}} \otimes  f^{(i-1)}_{w_{i}(1)}, \ 1 \otimes f^{(i-1)}_{w_{i}(2)}\right) \ .
 \]
Then $\beta^{(i)}_{w_{i}(2)}$ is a basis of the $R[[v]]$-module $\frakM^{(i)}_{w_i{}(2)}$, and ${}^{\overline{\varphi}}\beta^{(i-1)}_{w_{i}(2)}$ is a basis of the $R[[v]]$-module ${}^{\overline{\varphi}}\frakM^{(i-1)}_{w_{i}(2)}$.
\end{Prop}
The next definitions and proposition justify the name ``convenient bases" used above.
\begin{Deff} %\S 5.1.8 of KM
Given an eigenbasis $\beta$ of $\MM$ and an index $i \in \{0,1\}$, we define $A^{(i)}_{\beta} \in \mathrm{M}_{2}(R[[v]])$ as the matrix of the $R[[v]]$-linear map $\phi_{\MM, w_{i+1}(2)}^{(i)} : {}^{\overline{\varphi}}\frakM^{(i)}_{w_{i+1}(2)} \to \MM^{(i+1)}_{w_{i+1}(2)}$ when ${}^{\overline{\varphi}}\frakM^{(i)}_{w_{i+1}(2)}$ is endowed with the basis ${}^{\overline{\varphi}}\beta^{(i)}_{w_{i+1}(2)}$ and $\MM^{(i+1)}_{w_{i+1}(2)}$ is endowed with the basis $\beta^{(i+1)}_{w_{i+1}(2)}$. The matrix $A^{(i)}_{\beta}$ is called the {\it matrix of the partial Frobenius of $\MM$ at embbeding $i$ and with respect to $\beta$.}
\end{Deff}

\begin{Prop} % \S5.1.8 and 5.1.9 in KM
\label{prop:shape} 
Let $\MM \in Y^{\tau}(k)$ be a Kisin module over $k$. Then there exists an eigenbasis $\beta$ of $\MM$ such that each of the matrices $A^{(0)}_{\beta}$ and $A^{(1)}_{\beta}$ has one of the following forms (with $\bar{c}_{i,j} \in k$ and $\bar{c}^{*}_{i,j} \in k^{\times}$):
\begin{center}
\begin{tabular}{|c||c|c|c|}
\hline
$\widetilde{w}_i$ & $\mathfrak{t}$ & $\mathfrak{t}' $& $\mathfrak{w}$ \\
\hline
$A^{(i)}_{\overline{\beta}}$ & 
$\begin{pmatrix}
v \bar{c}^*_{1,1} & 0 \\
v \bar{c}_{2,1}  & \bar{c}^*_{2,2}
\end{pmatrix}$
 &  
 $\begin{pmatrix}
\bar{c}^*_{1,1} & \bar{c}_{1,2}  \\
0 & v \bar{c}^*_{2,2}
\end{pmatrix}$
  & 
  $\begin{pmatrix}
0& \bar{c}^*_{1,2}  \\
 v \bar{c}^*_{2,1} & 0 \\
\end{pmatrix}$
    \\
\hline
\end{tabular}
\end{center}
Moreover, the pair $(\widetilde{w}_{0}, \widetilde{w}_{1})$ in $\{ \mathfrak{t} , \mathfrak{t}' , \mathfrak{w} \}^2$ determined that way does not depend on the choice of the eigenbasis $\beta$, but only on the Kisin module $\MM$.
\end{Prop}

\begin{Deff}
\label{DefShape}
The pair $(\widetilde{w}_{0}, \widetilde{w}_{1})$ given by Proposition \ref{prop:shape} is called the {\it shape of $\MM$}.
\end{Deff}

\begin{rem}
The actual definition of the shape involves the extended Weyl group of $G$, but we choose to give this handy definition to avoid extra technicalities in this survey paper. For more details on  the way to define the shape via Weyl elements, we suggest to refer to \cite[Section 5.1.9]{KM}.
\end{rem}

\subsubsection{A deformation problem for Kisin modules}\label{sssec:deformKM}
In this penultimate subsection, we define a deformation problem for Frobenius-twist self-dual Kisin modules related to the deformation ring we introduced in Section \ref{sssec:def_Galois_potcrystype} for $C$-parameters.
We have a precise understanding of the form of the Kisin modules for any local Artinian $R$, generalising Proposition \ref{prop:shape}, which holds for the finite field $k$. This understanding leads to a very explicit description of the ring representing these deformations of Kisin modules (Theorem~\ref{Thm:expl_def_Kismod}). 
We first require a basis compatible with all the structures (descent data, polarisation) on the Kisin modules. We introduce this basis now.

 \begin{Deff}\label{Def:gauge_basis}
 Let $\frakM$ be  in $Y^{\tau} (R)$ and $(\widetilde{w}_0, \widetilde{w}_1)$ be the shape of its reduction $\frakM \otimes_R k$ to $k$.
 \begin{enumerate}[(i)]
\item A \emph{gauge basis} of $\frakM$ is an eigenbasis $\beta$ of $\frakM$ 
such that the matrices of the partial Frobenius $\left(A^{(0)}_{\beta}, A^{(1)}_{\beta} \right) $ have the form in the table below, prescribed by the shape of $\overline{\frakM}$ (with $c_{i,j}$ in $R$, $c^*_{i,j}$ in $R^\times$ and $c_{i,i}$ in $R \setminus R^\times$).
\begin{center}
\begin{tabular}{|c||c|c|c|}
\hline
$\widetilde{w}_i$ & $\mathfrak{t}$ & $\mathfrak{t}' $& $\mathfrak{w}$ \\
\hline
$A^{(i)}_{\beta}$ & 
$\begin{pmatrix}
(v + p)c^*_{1,1} & 0 \\
v {c}_{2,1}  &{c}^*_{2,2}
\end{pmatrix}$
 &  
 $\begin{pmatrix}
{c}^*_{1,1} &{c}_{1,2}  \\
0 & (v +p){c}^*_{2,2}
\end{pmatrix}$
  & 
  $\begin{pmatrix}
c_{1,1} & {c}^*_{1,2}  \\
 v {c}^*_{2,1} & c_{2,2} \\
\end{pmatrix}$
   \\
    & & & $c_{1,1}c_{2,2} = -pc_{1,2}^*c_{2,1}^*$ \\
\hline
\end{tabular}
\end{center}
\item If moreover the inertial type $\tau$ is Frobenius-twist self-dual, $R$ is local Artinian and $\iota$ is a polarisation on $\frakM$, a gauge basis $\beta$ is called \emph{compatible with $\iota$} if it satisfies : $\iota\left( \sigma^*\beta \right) = \beta^\vee$.
\end{enumerate}
 \end{Deff}

 \begin{Rem}
 \label{Rem:gauge_basis}
 \begin{itemize}
\item For $R = k$, a gauge basis is an eigenbasis such that the matrices of the partial Frobenius have the form in the table of Proposition \ref{prop:shape}.
\item For general local Artinian $R$, compatible gauge bases exist by analogues of \cite[Theorem 4.1]{S&S}  and \cite[Proposition 5.17]{KM}.
\end{itemize}
 \end{Rem}

From now on, we assume that $\tau$ is Frobenius-twist self-dual. Recalling that $s = \begin{psmallmatrix} 0 & 1 \\ 1 & 0 \end{psmallmatrix}$ is also a lift in $G$ of the non-trivial element of $W$, we can state the following result \cite[Lemma 5.18(i)]{KM}.
\begin{Lem}
Let $(\frakM, \iota) $ be an object of $Y^{\tau}_{\mathrm{pol}} (R)$ and let $\beta$ be a gauge basis of $\MM$ compatible with $\iota$. Then we have the following matrices relation: 
\[ \displaystyle A^{(0)}_{\beta} = (v+p)s \left( A^{(1)}_{\beta}\right)^{-\mathrm{t}}s \ . \]
\end{Lem}

\begin{rem}
\label{FormShapeModp}
Note that this relation implies in particular that the shape of an object $(\overline{\MM}, \overline{\iota}) \in Y^{\tau}_{\mathrm{pol}}(k)$ is necessarily of the form $(\widetilde{w}, \widetilde{w})$, i.e. that $\widetilde{w}_{0} = \widetilde{w}_{1}$ \cite[Lemma 5.18 (i)]{KM}.
\end{rem}
We are now ready to define the aforementioned nice deformation problem for Frobenius-twist self-dual Kisin modules over $k$. 
\begin{Deff} % 5.3.1 de KM
\label{DefinitionDefPbKisinModules}
Let $\tau$ be a $2$-generic principal series type that is Frobenius-twist self-dual. Let $(\overline{\MM}, \overline{\iota})$ be an object of $Y^{\tau}_{\mathrm{pol}}(k)$ and let $\overline{\beta}$ be a gauge basis on $\overline{\MM}$ that is compatible with $\overline{\iota}$. For any ring $R$ in $\mathrm{Art}(\OO)$, we define $D^{\tau, \overline{\beta}}_{\overline{\frakM}, \mathrm{pol}}(R)$ as the set of triples $\left( \frakM_R, \iota_R , \beta_R \right)$, where $(\MM_{R}, \iota_{R})$ denotes an object of $Y^{\tau}_{\mathrm{pol}}(R)$ that lifts $(\overline{\MM}, \overline{\iota})$, i.e. such that $\MM_{R} \otimes_{R} k \simeq \overline{\MM}$ and $\iota_{R} \otimes_{R} k = \bar{\iota}$, and $\beta_{R}$ a gauge basis of $\MM$ that is compatible with $\iota$ and lifts $\overline{\beta}$.
\end{Deff}

According to \cite[Section 5.3.1]{KM}, this defines a functor $D^{\tau, \overline{\beta}}_{\overline{\frakM}, \mathrm{pol}}$ that is representable by some object $R^{\tau, \overline{\beta}}_{\overline{\frakM}, \mathrm{pol}} \in \mathrm{Noe}(\OO)$. The latter ring has the following explicit description \cite[Theorem 5.19]{KM}, which can be deduced from the previous results once we note that the choice of a triple $\left( \MM_{R}, \iota_{R}, \beta_{R}\right)$ as above is equivalent to the choice of a matrix $A^{(1)}_{\beta_{R}}$ whose form is prescribed by the shape of $\overline{\MM}$ in the table of Definition \ref{Def:gauge_basis}.
\begin{Thm}
\label{Thm:expl_def_Kismod}
We keep the previous notation and assumptions. In particular, we let $(\widetilde{w},\widetilde{w})$ be the shape of $\overline{\mathfrak{M}}$. Then $R^{\tau, \overline{\beta}}_{\overline{\frakM}, \mathrm{pol}}$ is isomorphic to the ring $R^{\mathrm{expl}}_{\widetilde{w}}$ given by the following table:
\begin{center}
\begin{tabular}{|c||c|c|c|}
\hline
$\widetilde{w}$ & $\mathfrak{t}$ & $\mathfrak{t}'$ & $\mathfrak{w}$ \\
\hline
$R_{\widetilde{w}}^{\mathrm{expl}}$ & $\mathcal{O}[[c_{2,1},c_{1,1}^{\ast}, c_{2,2}^{\ast}]]$ &  $\mathcal{O}[[c_{1,2},c_{1,1}^{\ast}, c_{2,2}^{\ast}]]$ &  $\mathcal{O}[[c_{1,1},c_{2,2},c_{1,2}^{\ast}, c_{2,1}^{\ast}]]/(c_{1,1}c_{2,2} +p)$  \\
\hline
\end{tabular}
\end{center}
In particular, the form of $R^{\tau, \overline{\beta}}_{\overline{\frakM}, \mathrm{pol}}$ only depends on the shape of $\overline{\MM}$.
\end{Thm}

\subsubsection{Some consequences on deformations of $C$-parameters}
\label{ssec:expl_Galois_def}
Let $\rhobar : \Gamma_{\QQ_{p}} \to {}^CG$ be a tamely ramified $p$-modular $C$-parameter and let $\tau : \Inertia_{\QQ_{p^{2}}} \to \GL_{2}(\OO)$ be a principal series inertial type that is Frobenius-twist self-dual. Further assume that $\rhobar$ is $1$-generic, that $\tau$ is $2$-generic, and that there exists a Kisin module in $\overline{\MM} \in Y^{\tau}(k)$ such that $\Functor\left(\overline{\frakM} \right) \cong \mathrm{BC}(\rhobar)_{|\Gamma_{\Q_{p^2, \infty}}}$. In this case, we know from Lemma \ref{Lem:varK_trivial} that such a Kisin module is unique and that it comes with a natural polarisation $\overline{\iota}$. This ensures that it makes sense to define the {\it shape of $\rhobar$ with respect to $\tau$} as the shape of the corresponding Frobenius-twist self-dual Kisin module $(\overline{\MM}, \overline{\iota})$. Remark \ref{FormShapeModp} shows that this shape is of the form $(\widetilde{w},\widetilde{w})$, and our assumption allow us to use the previous results, and in particular the explicit formulae given by Theorem \ref{Thm:expl_def_Kismod}. This leads to the following result, which is a reformulation in this context of \cite[(5.3.2)]{KM}.
\begin{Thm}
\label{Thm:expl_def_Galois}
We have an isomorphism of formal series rings of the following form:
$$
R_{\bar{\rho}}^{\tau}[[S_1,S_2]] \cong R_{\widetilde{w}}^{\mathrm{expl}}[[T_1,T_2,T_3,T_4]] \ .
$$
\end{Thm}
Note that Theorem \ref{Thm:expl_def_Galois} does not completely describe $R_{\bar{\rho}}^\tau$, as it does not describe the image of this subring under the given isomorphism.
To fully understand the deformation problem $D^{\tau}_{\rhobar}$, we now plan to solve the following problem.
\begin{What}
Can we explicitly determine the deformation ring $R_{\bar{\rho}}^{\tau}$, without the variables $S_1, S_2$?
\end{What}

The isomorphism of \cite[(5.3.2)]{KM} is valid for deformations of the absolute Galois group of $\bQ_{p^f}$ (not only $\bQ_p$).
In this broader context, it identifies $R_{\bar{\rho}}^{\tau}[[S_1,\ldots, S_{2f}]] $ and $R_{\widetilde{w}}^{\mathrm{expl}}[[T_1,T_2,T_3,T_4]]$,
 $R_{\widetilde{w}}^{\mathrm{expl}}$ being the completed tensor product of $f$ explicit deformation rings (one for each factor (see (\ref{Eqn:factorKM})) of the Kisin module).
The additional formal smooth variables $S_1,\ldots, S_{2f}$ and $T_1,T_2,T_3,T_4$ correspond respectively to the gauge basis on the polarised Kisin module and the framing on the Galois representation.
Thus, we expect the $f$ pairs of variables $S_i$ to correspond to the $f$ factors of the ring~$R_{\widetilde{w}}^{\mathrm{expl}}$.

{\bf Acknowledgements} We are grateful for the opportunity to work together on this project and we wish to thank the organisers of the Women In Numbers Europe 3 conference. We also thank the referees for their careful reading and their valuable comments. The first author was partially supported by the ANR projects PerCoLaTor (ANR-14-CE25-0002-01) and GeRepMod (ANR-16-CE40-0010-01). The second author was partially supported by the ANR projects CLap-CLap (ANR-18-CE40-0026) and FLAIR (ANR-17-CE40-0012).
The third author has received funding from the European Research Council (ERC) under the European Union's Horizon 2020 research and innovation programme (grant agreement No. 714405) while working on this project. The fourth named author was supported by the Engineering and Physical Sciences Research Council~[EP/L015234/1], through the EPSRC Centre for Doctoral Training in Geometry and Number Theory (the London School of Geometry and Number Theory) at University College London.

\bibliographystyle{plain}
\bibliography{winebib}
\end{document}